\documentclass[11pt]{article}
\usepackage{amsbsy}
\usepackage{amssymb,amsthm, amsmath, latexsym}
\usepackage{mathrsfs}
\usepackage[english]{babel}
\usepackage{color}
\usepackage[
    colorlinks=true,
    citecolor=red
]{hyperref}
\usepackage{amsmath}
\usepackage{mathrsfs}
\usepackage{graphicx,amsmath,mathtools}
\usepackage{mathrsfs}
\usepackage{tikz}
\usepackage{wrapfig}
\usepackage{mathtools}
\usepackage{esint}
 \usepackage{stmaryrd}
\usepackage{amssymb,url,xspace}

\newtheorem{theorem}{Theorem}[section]
\newtheorem{proposition}{Proposition}[section]

\newtheorem{lemma}{Lemma}[section]
\newtheorem{corollary}{Corollary}[section]
\newtheorem{definition}{Definition}[section]

\numberwithin{equation}{section} \numberwithin{theorem}{section}
\numberwithin{proposition}{section} \numberwithin{lemma}{section}
\numberwithin{corollary}{section}
\numberwithin{definition}{section} \numberwithin{remark}{section}

 \newcommand{\T}{\S\kern .15em\relax }
 
\newcommand{\R}{\mathbb{R}}
\newcommand{\N}{\mathbb{N}}

\title{Boundary regularity for the polyharmonic Dirichlet problem}
  \author{{\sc Antoine Lemenant}\footnote{Universit\'e de Lorraine, CNRS, Institut Elie Cartan de Lorraine, BP 70239 54506 Vand\oe uvre-l\`es-Nancy Cedex, France ({\tt antoine.lemenant@univ-lorraine.fr}).}~\footnote{Institut Universitaire de France (IUF)}
\text{ and } {\sc Rémy Mougenot}\footnote{Universit\'e de Lorraine, CNRS, Institut Elie Cartan de Lorraine, BP 70239 54506 Vand\oe uvre-l\`es-Nancy Cedex, France ({\tt remy.mougenot@univ-lorraine.fr}).} 
}

\date{\today}

\begin{document}

\maketitle
 
\begin{abstract}  In this paper we prove that any solution of the $m$-polyharmonic Poisson equation in a Reifenberg-flat domain with homogeneous Dirichlet boundary condition, is $\mathscr{C}^{m-1,\alpha}$ regular up to the boundary. To achieve this result we extend the  Nirenberg  method of translations to operators of arbitrary order, and then use some Mosco-convergence tools developped in a previous paper \cite{GLM23}. 
\end{abstract}

 \vspace{1cm}

\tableofcontents
 
 \newpage

\section{Introduction}

This paper is devoted to the  regularity up to the boundary,  for solutions of a general elliptic PDE of order $2m$ with constant coefficients.  A prototype   is, for instance,  the following $m$-polyharmonic Poisson equation:
\begin{eqnarray}
\left\{
\begin{array}{c}
(-\Delta)^{m}u=f \quad \text{ in } \Omega \\
u \in H^m_0(\Omega).
\end{array}
\right. \label{probP0}
\end{eqnarray}

For any bounded domain $\Omega$ and  $f\in L^2(\Omega)$, it is easy to see that there exists a unique solution to the above problem. In the following, we also impose a mild regularity condition on the domain $\Omega$. Specifically, we consider the relatively broad class of Reifenberg-flat domains. This concept encompasses, in particular, Lipschitz domains and has been utilized in numerous previous studies on the regularity of elliptic PDEs, as  for instance 
\cite{
reifenberg4,
reifenberg3,
reifenberg8,
zbMATH05978487,
reifenberg5,
reifenberg9, 
reifenberg2,
reifenberg1, 
LMS13,
LS13,
reifenberg6,
chinois,
reifenberg7}, to mention a few.  Let us first introduce some notation.


For every $r>0$ and $x\in \R^N$, $B(x,r)$ is the open ball of radius $r$ and centered at $x$, $B_r:=B(0,r)$. Moreover for $\lambda \in [-1,1]$ we will denote by  $B^\lambda(x,r)$  the truncated ball of level $\lambda$ defined by
\begin{align*}
    B^\lambda(x,r) := B(x,r) \cap \left\{y\in \R^N \; \middle| \; y_N>\lambda r\right\}.
\end{align*}
We also denote  $B^+(x,r):=B^0(x,r)$ the upper half-ball centered at $x$.

 \begin{definition}[\textnormal{Reifenberg flat domain}]\label{defreif}
Let $r_0>0$ and $\varepsilon_0\in(0,1]$. A bounded open set is $(\varepsilon_0, r_0)-$Reifenberg flat if for every $x\in \partial \Omega$ and $ r\in [0, r_0]$, there exists a rotation $R_x$ centered at $x$ such that
        \begin{align*}
    B^{\varepsilon_0}(x,r) \subset R_x(\Omega) \cap B(x,r) \subset B^{-\varepsilon_0}(x,r).
\end{align*}
\end{definition}

This definition is equivalent to consider a hyperplan $\mathscr{P}_x$ containing $x$ such that 
\begin{align*}
    d_\mathscr{H}(\partial \Omega \cap B(x,r), \mathscr{P}_x \cap B(x,r)) \leq \varepsilon_0r, \quad r\leq r_0,
\end{align*}
where $d_\mathscr{H}$ is the Hausdorff distance, and to suppose that
    \begin{align*}
        B(x,r) \cap \left\{ y \in \R^N\; \middle| \; {\rm }dist(y, \mathscr{P}_x)\geq2\varepsilon_0r\right\}
    \end{align*}
    has two connected components, one lying in $\Omega$, the other one lying in $\R^N \backslash \Omega$, see \cite{R60}.

    \medskip
     
    The purpose of this paper is to investigate the boundary regularity for the Problem \eqref{probP0} in Reifenberg flat domains.  Actually, it applies to more general operators of order $2m$   of the form
    \begin{align}
        \mathscr{A} := (-1)^m \sum_{\vert \alpha \vert =\vert \beta \vert= m}  a_{\alpha,\beta}\partial^{\beta} \partial^{\alpha}, \label{ellipticOperator0}
    \end{align}
    where $a_{\alpha,\beta} \in \R$ are real coefficients. We assume that $\mathscr{A}$ symmetric, i.e. $a_{\alpha,\beta}=a_{\beta,\alpha}$, and  is elliptic in the sense that there exists $C>0$ such that for all vectors $\xi=(\xi_{\alpha})_{|\alpha|=m}$ with $\xi_\alpha\in \R$ it holds,
    \begin{eqnarray}
    \sum_{\vert \alpha \vert =\vert \beta \vert= m}  a_{\alpha,\beta}\xi_\alpha \xi_{\beta}\geq C \sum_{\vert \alpha \vert =m} |\xi_\alpha|^2. \label{ellipticPartial0}
    \end{eqnarray}
      
   Here is our main result.
    
        \begin{theorem}\label{main}
    Let $\alpha \in (0,1)$ and $q\geq 2$ be such that $mq\geq N$ if $2m<N$, and $q=2$ otherwise. There exists $\varepsilon_0 \in (0,1)$ and $r_0 \in (0,1]$ such that for every $(\varepsilon_0,r_0)-$Reifenberg-flat domain $\Omega\subset \R^N$, for every function $f\in L^q(\Omega)$, if $u \in H^m_0(\Omega)$ is the weak solution of $\mathscr{A}(u) = f$, then $u\in \mathscr{C}^{m-1,\alpha}(\overline{\Omega})$ and 
    $$\|u\|_{\mathscr{C}^{m-1,\alpha}(\overline{\Omega})} \leq C \|f\|_{L^q(\Omega)},$$
    where $C>0$ depends on $N$,  $A$,  $\alpha$, $\Omega$ and $m$.
\end{theorem} 

What is somehow surprising in Theorem \ref{main}, is that we are able to obtain fairly strong regularity up to the boundary, in the class $\mathscr{C}^{m-1,\alpha}$, although $\partial \Omega$ enjoys only very poor regularity, not better than $\mathscr{C}^{0,\gamma}$.  Indeed, a Reifenberg-flat domain is less regular that Lipschitz, it could even  have a  fractal boundary. The high regularity of solution in this poor regularity class of domains is possible due to the fact that $u\in H^{m}_0(\Omega)$ thus vanishes at the boundary, as well as all its $k$th derivatives, up to the order $m-1$, leaving a chance to be more regular up to the boundary, than the boundary itself.

For instance our result immediately implies that in dimension $3$ and in the case of the bi-laplacian, if $u\in H^2_0(\Omega)$ and $\Delta^2 u \in L^2(\Omega)$, then $u\in\mathscr{C}^{1,\alpha}(\overline{\Omega})$, while $\Omega$ is only Reifenberg-flat.

Notice that for the interior regularity, if $\Delta^m(u)=f$ with $f\in L^q$ then $\Delta(\Delta^{m-1})(u)=f$ so that $\Delta^{m-1}u\in W^{2,q}$ by classical Calderon-Zygmund estimates on the standard Laplacian. Iterating this $m$ times we arrive to the fact that $u\in W^{2m,q}$.   Hence, if $mq>N$,  the Sobolev embedding theorem implies that $\nabla^m u \in \mathscr{C}^{0,\alpha}$, or in other words $u \in \mathscr{C}^{m,\alpha}$. Of course this argument cannot be applied at the boundary because the Dirichlet boundary condition is not preserved while taking derivatives. But it is interesting to notice that comparing to the statement of Theorem~\ref{main}, since we arrive at the boundary up to   $u \in \mathscr{C}^{2m-1,\alpha}$ and not $\mathscr{C}^{2m,\alpha}$, we loose one derivative  with respect to what we could expect in the interior. This is due to the weak regularity of the boundary that prevents us to go beyond.

Let us furthermore emphasis that our result is sharp. Indeed, for instance for $m=1$ it is easy to construct a harmonic function with Dirichlet boundary condition on a cone fo aperture $\pi+\varepsilon_0$, that is in $\mathscr{C}^{0,\alpha}$ but not Lipschitz (see for instance Remark 18 in \cite{LMS13}).

A similar result as our Theorem \ref{main} has been partially obtained earlier in \cite{LS13} in the case $m=1$, i.e for the standard Laplacian, where a priori estimates $u \in L^p(\Omega)$ and $f \in L^q(\Omega)$ leads to the $\mathscr{C}^{0,\alpha}(\overline{\Omega})$ regularity,  provided that $\Omega$ is ``flat" enough. Our result includes the one in \cite{LS13}, and is even more general for the case of the standard Laplacian ($m=1$) because we do not assume any a priori bound on $u \in L^p(\Omega)$ as in \cite{LS13}. The technique is also different since in \cite{LS13} a monotonicity formula has been used, which is not available for the polyharmonic operator.

 The result in \cite{LS13} has also been extended for general elliptic operators of order 2 in  \cite{chinois}, using some barrier arguments. This technique cannot be applied   for operators  of highier order since it relies on the maximum principle.

 Notice that in the recent preprint \cite{prade2025boundaryregularitynonlocalelliptic}, a similar regularity result has been obtained for the fractional Laplacian $(-\Delta)^s$. For the particular case $s=1$, the result in   \cite{prade2025boundaryregularitynonlocalelliptic} is coherent with the  particular case of our result with $m=1$.

 For  the parabolic case, an $L^p$ theory has been developed in \cite{zbMATH05978487} in a general context of varying $BMO$ coefficients, from which one can probably deduce  a similar  statement as our main result, after an application of  the Sobolev embedding theorem. However, the technique employed in the present paper is completely different from the one in \cite{zbMATH05978487}, and somehow simpler.

It is worth mentioning that many works by S. Mayboroda and co-autors are dedicated to the study of the bi-harmonic operator in a general class of wild domains, including in particular Reifenberg-flat domains (see for instance \cite{M1,M2,M3,M4}). The context in these works is a bit different  since they are interested in the polyharmonic measure thus focus on the  delicate problem $\Delta^2u=0$  in $\Omega$ with $u=f$ on the boundary. In the present paper instead, we are looking for  $u=0$ on the boundary with a source $\Delta^2u=f$ in $\Omega$, which is slightly different.

 Our approach, which is completely self-contained in this paper, uses variational technics that are well adapted for solutions with homogenous boundary conditions, that might not work so well with an inhomogeneous condition. On the other hand, it provides some very precise energy estimates at the boundary in that particular case, that one could probably not obtain so easily with the general tools developed for the polyharmonic measure.

 We also believe that our result could be extended to more general operators of divergence type with H\"older continuous coefficients, but we do not pursue this generalization for sake of simplicity, restricting ourselves to constant coefficients.


\medskip

{\bf Ideas of proof and structure of the paper.} Let us now describe the ideas behind the proof of Theorem \ref{main}. The main point is to arrive to a decay behavior of the energy at the boundary for a poly-harmonic function, of the following type: there exists $b,C,r_0>0$ such that for all $x\in \partial \Omega$ and $r\leq R \leq r_0$, (in the sequel, a function in $H^m_0(\Omega)$ is considered as being extended by $0$ outside $\Omega$)
    \begin{align}\label{decayE}
          \int_{B(x,r )}\vert \nabla^m u \vert^2 dx \leq C\left(\frac{r}{R}\right)^{N-b}\int_{B(x,R)}\vert \nabla^m u \vert^2 dx, \quad r \in(0,R).
    \end{align}
   If  \eqref{decayE} holds, then one arrives to say that  $u \in \mathscr{C}^{m-1,\alpha}(\overline{\Omega})$ by standard Campanato theory.

Now  to obtain  \eqref{decayE} at the boundary point of a  Reifenberg-flat domain, the first step is to prove that it holds true when the boundary is a hyperplane. Once this is established, then we can argue by compactness and obtain that it still holds true when the boundary is close enough to a hyperplane. This is done in Proposition \ref{decayE}, which uses in particular the Mosco-convergence theorem for higher order operators contained in  \cite{GLM23}. Since a Reifenberg flat domain is $\varepsilon_0$-close to a hyperplane at any scale, we can then conclude that the same decay holds  for a solution in a Reifenberg-flat domain, leading to our main $\mathscr{C}^{m-1,\alpha}$ result.

Therefore, a fairly big part of this paper is actually  devoted to the proof of the energy decay property \eqref{decayE}, in the case of a perfect flat boundary, i.e. when $\partial \Omega$ coincides with a hyperplane.

For the standard regularity theory for operators of order $2m$, one usually quote \cite{ADN1,ADN2}, based on Fourier analysis. However, we could not find in the literature, the exact statement that is needed for our purposes. Instead,  we provide in this paper a complete and self-contained proof, based on the Nirenberg translation method, that allows us to control precisely all the constants.

 The interesting fact is that  we have to adapt this method in a non trivial way for higher order operators, compared to what is commonly known for the Laplace operator. Up to our knowledge,  we believe this argument to be new and interesting for its own. Indeed, let us explain the difficulty.

For the recall, a usual way to treat the problem of boundary regularity in the flat case, for instance when  $\Omega=\{e_N>0\}$, is to consider a  direction of translation  $h\in \R^N\backslash \{0\}$ and then introduce  the translated function 
    $$ D_{h}  u(x) = \frac{u(x+h)-u(x)}{ \vert h \vert }, \quad x\in \Omega.$$
    The main idea is the following: up to localize with a cut-off function and provided that  $h$ goes in the horizontal directions (i.e. for $h=e_i$ and $1\leq i \leq N-1$),  the function $D_h   u$ is an admissible test  function in $H^m_0(\{e_N>0\})$ in the weak formulation of the problem. For the polyharmonic problem  this leads to an apriori bound  of the type 
    $$\|D_h \nabla^mu\|_2 \leq C.$$
    Since $C$ is uniform in $h$, we easily conclude by taking $h\to 0$ that actually $u$ admits one more derivative in all the horizontal directions. We can continue inductively and actually prove that all the derivatives of the form $\partial^\alpha \nabla^m u$ are in $L^2$, where $\alpha$ are multi-indices involving only horizontal directions. Then we only miss the vertical direction, in which the translation is not admissible as a test function.

    For the standard Laplacian,  the usual trick works as follows: we  can recover the vertical derivatives by use of the operator. Indeed, if $-\Delta u=f$ then
   $$-\partial_N^2 u = f +\sum_{i=1}^{N-1} \partial^2_i u,$$
   and this is how we obtain that $\partial_N^2 u \in L^2$ because  we know that all the $\partial^2_i u$ for $1\leq i \leq N-1$ are in $L^2$.  We then conclude that $u \in H^2$, which is a gain of regularity compared to the original information that $u \in H^1$.

For higher order operators $\mathscr{A}$, we cannot recover the estimate on $\partial_N^{2m}u$ so easily because we need to pass from $\nabla^m u \in L^2$ to $\nabla^{2m} u \in L^2$: in other words we have to win $m$ derivatives, and not only one.  Moreover, the operator is more complicated and involves multiples derivatives of different order.

As a result, our proof works by induction on the order  of the operator. We first establish  the regularity theorem  for $m=1$, which is the case for instance for the standard Laplace operator. Then assuming that the theorem is true for operators $\mathscr{A}$ of order $2(m-1)$, we prove that it holds true for operators of order $2m$. The main tool is to write  a general operator $\mathscr{A}$ of order $2m$ as the sum (see Lemma \ref{decomposition0}):
 \begin{eqnarray}
\mathscr{A}(u)=      \mathscr{B}(u)    -\partial_N\mathscr{C}(\partial_N u), \label{dec}
\end{eqnarray}
where $\mathscr{B}(u)$ contains at most  $m$-derivatives in the $e_N$ direction, so that it belongs to $L^2$, and $\mathscr{C}$  is an elliptic operator of order $2(m-1)$, on which we can apply our induction hypothesis. Indeed, since $\mathscr{A}(u)\in L^2$ by assumption, we deduce that $\partial_N\mathscr{C}(\partial_N u) \in L^2$. Then by use of a certain Poincar\'e inequality, this actually yields $\mathscr{C}(\partial_N u) \in L^2$, and thus $\partial_N u \in H^{2(m-1)}$ thanks to our induction hypothesis (i.e. using that the Theorem is true for operators of order $2(m-1)$). All this implies that $u \in H^{2m-1}$.

Now we  are still missing one last derivative because we need to achieve $H^{2m}$ instead of $H^{2m-1}$. This will be done by use of  the  operator, for a last time again.  Indeed, at this stage of the proof, we have a control on  all the derivatives of the form $\partial^{\alpha}u$ with $|\alpha|=2m$, provided that  $\alpha_N\leq 2m-1$. In other words, to conclude that $u \in H^{2m}$ we only miss the last derivative $\partial^{2m}_N u$.  But this term appears in  $\mathscr{A}$ only once. Namely, we have 
\begin{eqnarray}
\mathscr{A}=(-1)^m a_{me_N, me_N} \partial^{2m}_N  + \mathscr{E}
\end{eqnarray} 
where $\mathscr{E}$ contains at most $2m-1$ derivatives in the $e_N$ direction. This is how we get  $\partial^{2m}_N u \in L^2$, and finally conclude that $u \in H^{2m}$, as desired.

All this leads to the regularity result that is contained in Theorem \ref{theoremederegularite2}, that we state here in the introduction to   enlight   what we have obtained with this method.

 \begin{theorem}[Regularity with  flat boundary]\label{theoremederegularite20}
    Let $m\in \N^*$. Let $\mathscr{A}$ be an  operator of order $2m$ of the form \eqref{ellipticOperator0} and satisfying the ellipticity condition \eqref{ellipticPartial0}. Let  $u \in H^m(B(0,1))$ be a weak solution for $ \mathscr{A}u=f$ in $B^+(0,1)$ with  $f \in H^\ell(B^+(0,1))$, and such that $u=0$ in $B(0,1)\cap \{x_N<0\}$. Then $u \in H^{2m+\ell}(B^+(0,1/2))$ and
    $$\|u\|_{H^{2m+\ell}(B^+(0,\frac{1}{2}))}\leq C(\|\nabla^m u\|_{L^{2}(B^+(0,1))} + \|f\|_{H^\ell(B^+(0,1)}),$$
 where $C>0$ is a constant that depend  on    $N$,  $\max_{\alpha,\beta} |a_{\alpha,\beta}|$, $l$ and $m$.
  \end{theorem}

We also have written the interior case, analogous to Theorem \ref{theoremederegularite20} (see Theorem \ref{theoremederegularite}). As a matter of fact, by reasoning with local charts,  our paper provides  a self-contained and full regularity theory for the Dirichlet problem of $2m$-order in a smooth domain.

A funny observation is that our method really needs to work with general operators of the form \eqref{ellipticOperator0}, and does not simplify for the special case of $(-\Delta)^m$. Indeed,  while performing the  induction argument on $(-\Delta)^m$ we would need to decompose it as the sum \eqref{dec}, in which an operator of lower order appears, which is not $(-\Delta)^{m-1}$, but more general, like in   \eqref{ellipticOperator0}.

Then in the next section we use this regularity result to establish a decay property of the type \eqref{decayE} for a poly-harmonic function at a boundary point of a Reifenberg-flat domain. This is obtained by use of a compactness argument (see Proposition \ref{propdecroissanceenergiedebut}). Once this is done, we can use a ``poly-harmonic'' replacement argument  to derive a similar decay property for a solution with a second member. This is done in Proposition \ref{decroissanceenergiereifenbergplat}, that leads to the proof of our  main result, Theorem \ref{main}.

\section{Preliminaries}

   As usual we denote by $H^m(\Omega)$ the Sobolev space endowed with the norm
    $$\|u\|_{H^m(\Omega)}^2=  \sum_{|\alpha|\leq m}\|\partial^\alpha u\|_2^2$$
    and we denote by $H^{m}_0(\Omega)$ the closure of $C^\infty$ function with compact support in $\Omega$. Since a Reifenberg-flat domain enjoys the so-called corskcrew condition, it follows that
    \begin{eqnarray}
     H^{m}_0(\Omega)=\{u \in H^m(\R^N) \text{ s.t. } u=0 \text{ a.e.  on  } \R^N \setminus \Omega\}. \label{netrusov}
     \end{eqnarray} 
 This is actually Corollary 5.1. in \cite{GLM23}, and we refer to \cite{GLM23} for more details. As a consequence, in all the paper we will consider $u\in H^1_0(\Omega)$ as begin a function of $H^m(\R^N)$, extended by $0$ on $\Omega^c$.

 We will also use \eqref{netrusov} implicitly while dealing with an mixed boundary value problem in a half ball.  For instance in Section \ref{boundaryReg} we will consider $u \in H^m(B(0,1))$ such that $u=0$ a.e. on $B(0,1)\cap \{ x_N<0\}$. It is known (see for instance  \cite{GLM23}) that such a function can be approximated by $v_n\to u$ in $H^m(B(0,1))$ such that ${\rm supp}(v_n)\subset B^+(0,1)$ for all $n$. This fact will also be used in the proof of Proposition \ref{propdecroissanceenergiedebut}.

\subsection{Multi-index and Leibniz formula} A multi-index is a vector $\alpha \in \N^N$. The length of $\alpha$ is denoted by $|\alpha|$ which is the sum of all the coefficients $\alpha_i$. The partial derivative $\partial^\alpha$ means that we take all derivatives $\partial_{x_i}^{(\alpha_i)}$. We denote by $\alpha!$ the number $\alpha_1!\alpha_2!\dots \alpha_N!$ and we say that $\alpha \leq \beta$ when $\alpha_i\leq \beta_i$ for all $1\leq i\leq N.$

 If $u,v$ are two smooth functions then the Leibniz formula says
\begin{align}\label{leibniz}
    \partial^\alpha (uv) &= \sum_{\beta\leq \alpha}\frac{\alpha!}{\alpha! (\alpha-\beta)!}\partial^{\beta}u\partial^{\alpha-\beta}v.
\end{align}

\subsection{Elliptic operator of order $2m$} Let $m\in \N^*$. We consider a general  operator $\mathscr{A}$  of order $2m$ of the form
    \begin{align}
        \mathscr{A} := (-1)^m \sum_{\vert \alpha \vert =\vert \beta \vert= m}  \partial^\beta(a_{\alpha,\beta} \partial^{\alpha})=(-1)^m \sum_{\vert \alpha \vert =\vert \beta \vert= m}  a_{\alpha,\beta}\partial^{\beta} \partial^{\alpha}, \label{ellipticOperator}
    \end{align}
    where $a_{\alpha,\beta} \in \R$ are real coefficients. We assume that $\mathscr{A}$ symmetric, i.e. $a_{\alpha,\beta}=a_{\beta,\alpha}$, and  is elliptic in the sense that there exists $C>0$ such that for all vectors $\xi=(\xi_{\alpha})_{|\alpha|=m}$ with $\xi_\alpha\in \R$ it holds,
    \begin{eqnarray}
    \sum_{\vert \alpha \vert =\vert \beta \vert= m}  a_{\alpha,\beta}\xi_\alpha \xi_{\beta}\geq C \sum_{\vert \alpha \vert =m} |\xi_\alpha|^2. \label{ellipticPartial}
    \end{eqnarray}
    
 \subsection{Weak Formulation}
 For  $f\in L^2(\Omega)$, the existence and uniqueness for the problem
\begin{eqnarray}
\left\{
\begin{array}{c}
  \mathscr{A}u=f \quad \text{ in } \Omega \\
u \in H^m_0(\Omega)
\end{array}
\right. \label{probP}
\end{eqnarray}
can be obtained considering minimizers of

$$\min_{u\in H^m_0(\Omega)} \int_{\Omega} \sum_{\vert \alpha \vert =\vert \beta \vert= m} a_{\alpha,\beta} \partial^\alpha u \partial^\beta u \;dx - \int_{\Omega} uf \;dx.$$

The weak formulation of this problem is 
\begin{eqnarray}
 \int_{\Omega} \sum_{\vert \alpha \vert =\vert \beta \vert= m} a_{\alpha,\beta} \partial^\alpha u \partial^\beta \varphi \;dx = \int_{\Omega} \varphi f \;dx \quad \text{ for all } \varphi \in H^m_0(\Omega). \label{weakForm}
 \end{eqnarray}

\subsection{Matrix formulation} For $u \in H^m(\Omega)$ we denote by $\nabla^m u$ the vector of all derivatives of order $m$, i.e. $\nabla^m u=(\partial^\alpha u)_{|\alpha|=m}$. Then the weak formulation \eqref{weakForm} can also be written as follows,
\begin{eqnarray}
 \int_{\Omega} A \nabla^m u \cdot \nabla^m \varphi \;dx = \int_{\Omega} \varphi f \;dx \quad \text{ for all } \varphi \in H^m_0(\Omega), \label{weakForm2}
 \end{eqnarray}
 with  $A=(a_{\alpha,\beta})$.   
 

\subsection{Poly-harmonic operators}

In particular, our class of operators contains   the poly-harmonic operator $(-\Delta)^m$ (see \cite{gazzola}) by taking the particular case 
 $a_{\alpha,\beta}=\frac{m!}{\alpha!}\delta_{\alpha,\beta}$ where $\delta_{\alpha,\beta}=1$ if $\alpha=\beta$ and $0$ otherwise. In this case,  
 \begin{eqnarray}
    \sum_{\vert \alpha \vert =\vert \beta \vert= m} a_{\alpha,\beta}\xi_\alpha \xi_{\beta}= \sum_{\vert \alpha \vert =m} \frac{m!}{\alpha!} |\xi_\alpha|^2, \notag
    \end{eqnarray}
and therefore it is clear that this operator satisfies the ellipticity condition \eqref{ellipticPartial}.

Even if the following fact will not be used in this paper, it is worth mentioning that  by use of a standard integration by parts valid in $H^m_0(\Omega)$, one can prove that  the associated bilinear form that appears  in the weak formulation for $(-\Delta)^m$ becomes 

\begin{eqnarray}
 \int_{\Omega} \sum_{\vert \alpha \vert =\vert \beta \vert= m} a_{\alpha,\beta} \partial^\alpha u \partial^\beta \varphi \;dx
 &=& 
 \left\{
 \begin{array}{cc}
 \int_{\Omega} \Delta^{\frac{m}{2}}u \Delta^{\frac{m}{2}}\varphi  \;dx& \text{ if } m \text{ is  even,}\\
  \int_{\Omega}\nabla \Delta^{\frac{m-1}{2}}u \cdot\nabla \Delta^{\frac{m-1}{2}}\varphi \;dx& \text{ if } m \text{ is  odd.}\\
 \end{array}
 \right. \notag
 \end{eqnarray}

\subsection{Decomposition of Elliptic operators}  In the sequel we will need the following Lemma, which is one of the key step in our proof.

\begin{lemma} \label{decomposition0}  Let  $\mathscr{A}$ be an operator of order $2m$ of the form \eqref{ellipticOperator} that satisfies the ellipticity condition \eqref{ellipticPartial}. Then  $\mathscr{A}$ can be splitted in two parts, $\mathscr{A}=      \mathscr{B} +    \mathscr{C}$, as follows
\begin{eqnarray}
 \mathscr{A}&:=&(-1)^m\sum_{|\alpha|=\vert \beta \vert = m} a_{\alpha,\beta} \partial^\beta \partial^\alpha \notag \\
 &=&\underbrace{(-1)^m\sum_{\substack{|\alpha|=\vert \beta \vert = m \\ \alpha_N=0 \text{ or } \beta_N=0}}\partial^\beta( a_{\alpha,\beta} \partial^\alpha)}_{ \mathscr{B} } \; +\; \underbrace{(-1)^m\sum_{\substack{|\alpha|=\vert \beta \vert = m \\ \alpha_N>0 \text{ and  } \beta_N>0}}\partial^\beta( a_{\alpha,\beta} \partial^\alpha)}_{ \mathscr{C}}, \notag  
 \end{eqnarray}
in such a way that 
$$ \mathscr{C}:= (-1)^m \sum_{\substack{|\alpha|=\vert \beta \vert = m \\ \alpha_N>0 \text{ and  } \beta_N>0}}\partial^\beta( a_{\alpha,\beta} \partial^\alpha)=\mathscr{D}\circ (-\partial_N^2),$$
where $\mathscr{D}$ is elliptic, as  an operator of  order $2m-2$.
 \end{lemma}
\begin{proof} By definition of $\mathscr{C}$  we can factorize by $\partial_N^2$ as follows:
$$ \mathscr{C}:= (-1)^m \sum_{\substack{|\alpha|=\vert \beta \vert = m \\ \alpha_N>0 \text{ and  } \beta_N>0}}a_{\alpha,\beta} \partial^\beta  \partial^\alpha=(-1)^{m-1} \sum_{|\alpha|=\vert \beta \vert = m-1 }a_{\alpha+e_N,\beta+e_N} \partial^\beta  \partial^\alpha(- \partial_N^2),$$
and therefore the operator $\mathscr{D}$ of the statement of the Lemma is given by 
$$\mathscr{D}:=(-1)^{m-1} \sum_{|\alpha|=\vert \beta \vert = m-1 }a_{\alpha+e_N,\beta+e_N} \partial^\beta  \partial^\alpha.$$
We have to check that this operator is elliptic, as an operator of order $2m-2$. For that purpose we will use the ellipticity of  $\mathscr{A}$ on suitably chosen vectors $\xi$.  More precisely, let  $\eta=(\eta_\alpha)_{|\alpha|=m-1}$ be any given vector indexed over all multi-indices of lenght $m-1$.   We complete the vector $\eta$  by defining a vector $\xi$ indexed by multi-indices of length $m$, by writing now    when $|\alpha|=m$, 
$$
\xi_\alpha:=
\left\{
\begin{array}{ll}
\eta_{\alpha-e_N} &\text{ if }  \alpha_N >0 \\
0 &\text{ if  } \alpha_N =0.
 \end{array}
 \right.
 $$
 By testing the ellipticity of $\mathscr{A}$ with this vector $\xi$ we get 
 \begin{eqnarray}
  \sum_{|\alpha|=\vert \beta \vert = m} a_{\alpha,\beta} \xi_\beta \xi_\alpha= \underbrace{\sum_{\substack{|\alpha|=\vert \beta \vert = m \\ \alpha_N=0 \text{ or } \beta_N=0}} a_{\alpha,\beta} \xi_\alpha \xi_{\beta}}_{ =0 } \; +\;   \sum_{\substack{|\alpha|=\vert \beta \vert = m \\ \alpha_N>0 \text{ and  } \beta_N>0}}  a_{\alpha,\beta} \xi_\alpha \xi_\beta \geq C  \sum_{\vert \alpha \vert =m} |\xi_\alpha|^2, \notag  
 \end{eqnarray}
 which directly gives 
$$  \sum_{ |\alpha|=\vert \beta \vert = m-1 }  a_{\alpha+e_N,\beta+e_N} \eta_{\alpha} \eta_{\beta} \geq C  \sum_{\vert \alpha \vert =m-1} |\eta_\alpha|^2,$$
and this proves that $\mathscr{D}$ is elliptic, as an operator of order $2m-2$.
 \end{proof}
 
\medskip

\subsection{Poincar\'e type inequalities} 
The following simple lemma will be needed in the proof of Theorem~\ref{theoremederegularite}. The proof is elementary, by integrating on vertical rays. 

\begin{lemma}\label{inegalitedeplacementmasse} For $r>0$ we denote by $Q_r\subset \R^N$ the cube centered at the origin of side $r$, i.e.,
$$Q_r:=\{x\in \R^N \; | \;  \forall 1\leq i \leq N, \; |x_i|\leq r\}.$$
For $\lambda \in (0,1)$ we also introduce 
 $$Q_r^\lambda := Q_r \cap \{x_N> \lambda r \}.$$ 
Assume that  $v \in L^2(Q_\lambda)$ for all $\lambda \in (0,1)$ and that furthermore $v$ satisfies 
$$\Vert \partial_N v \Vert_{L^2(Q_r)} <+\infty.$$
Then  $v \in L^2(Q_r)$ and for all $\lambda \in (0,3/4)$ it holds:    
\begin{eqnarray}
        \Vert v \Vert_{L^2(Q_r)} \leq 4\Vert v \Vert_{L^2(Q_r^\lambda)}+ 3r\Vert \partial_N v \Vert_{L^2(Q_r)}. \label{ammontre}
    \end{eqnarray} 
\end{lemma}%

\begin{proof} We first assume that $v\in C^\infty(\overline{Q_r})$. We will use the notation $x=(x',x_N)$ for a point in $\R^N$. Then integrating along the $e_N$ direction  yields,  for every $(x',t)\in Q_r$,
$$v(x',t)=v(x',s)+\int_{s}^{t} \partial_N v(x',z) \;dz.$$
We then integrate  over $\lambda r\leq s \leq r$, 
$$(1-\lambda)r v(x',t)=\int_{\lambda r}^r v(x',s) ds+\int_{\lambda r}^r \int_{s}^{t} \partial_N v(x',z) \;dz\,ds.$$
Taking the square and using that $(a+b)^2\leq 2(a^2+b^2)$ it holds
$$(1-\lambda)^2r^2 v(x',t)^2\leq 2\left(\int_{\lambda r}^r v(x',s) ds\right)^2+2\left((1-\lambda)r \int_{-r}^{r} |\partial_N v(x',z)| \;dz\right)^{2},$$
which becomes, by Cauchy-Schwarz inequality,
\begin{eqnarray}
(1-\lambda)^2r^2 v(x',t)^2&\leq& 2(1-\lambda)r\int_{\lambda r}^r v(x',s)^2 ds+4(1-\lambda)^2r^3 \int_{-r}^{r} (\partial_N v(x',z))^2 \;dz, \notag 
\end{eqnarray}
or differently,
$$ (1-\lambda)r v(x',t)^2\leq  2\int_{\lambda r}^r v(x',s)^2 ds+4(1-\lambda)r^2 \int_{-r}^{r} (\partial_N v(x',z))^2 \;dz. $$
We then  integrate over $x' \in Q_{r}'$ (which is the cube in $\R^{N-1}$), arriving to
\begin{eqnarray}
(1-\lambda)r \int_{Q_r'}v(x',t)^2 \;dx'&\leq& 2 \int_{Q_r^\lambda} v(x)^2 dx+ 4(1-\lambda)r^2 \int_{Q_r} (\partial_N v(x))^2 \;dx.  \notag
\end{eqnarray}
A last integration over $-r\leq t\leq r$ finally yields
\begin{eqnarray}
(1-\lambda)r \int_{Q_r}v(x)^2 \;dx&\leq& 4r \int_{Q_r^\lambda} v(x)^2 dx+ 8(1-\lambda)r^3 \int_{Q_r} (\partial_N v(x))^2 \;dx,  \notag
\end{eqnarray}
or equivalently,
\begin{eqnarray}
 \int_{Q_r}v(x)^2 \;dx&\leq& \frac{4}{1-\lambda} \int_{Q_r^\lambda} v(x)^2 dx+ 8r^2 \int_{Q_r} (\partial_N v(x))^2 \;dx,  \notag
\end{eqnarray}
which proves the inequality \eqref{ammontre} in the case when $v \in C^\infty(\overline{Q_r})$, because $\lambda \leq 3/4$ by assumption.

Now let $v$ be as in the statement, i.e. $v \in L^2(Q_\lambda)$ for all $\lambda \in (0,1)$ and that furthermore $\Vert \partial_N v \Vert_{L^2(Q_r)} <+\infty$. We argue by approximation: for $\varepsilon>0$, we let $v_\varepsilon := v * \rho_\varepsilon$ where $\rho_\varepsilon$ is a standard mollifier. Then since $v \in L^2(Q_\lambda)$ for all $\lambda \in (0,1)$, we know that $v_\varepsilon \to v$ in $L^2(Q_\lambda)$ for all $\lambda \in (0,1)$ and by extracting a diagonal sequence we can assume that $v_\varepsilon$ converges to $v$ a.e. on $Q_r$. Then by applying \eqref{ammontre}  to $v_\varepsilon$, and passing to the limit $\varepsilon\to 0$ we  get from  Fatou Lemma that
$$
 \Vert v \Vert_{L^2(Q_r)} \leq \liminf_{\varepsilon \to 0} \|v_\varepsilon\|_{L^2(Q_r)} \leq  4\Vert v \Vert_{L^2(Q_r^\lambda)}+ 3r\Vert \partial_N v \Vert_{L^2(Q_r)},
$$
which ends the proof.
\end{proof}

We will also need the following standard version of the Poincar\'e inequality, that follows from instance from \cite[Corollary 4.5.2, page 195]{Zi}.

\begin{lemma}\label{poincare}  If $v \in H^m(B(0,1))$ is such that $v=0$ a.e. in $B(0,1)\cap \{x_N<0\}$, then
    \begin{align*}
    \|v\|_{H^m(B(0,1))} \leq  C \|\nabla^m v \|_{L^2(B(0,1))},
        \end{align*} 
        where $C>0$ depends only on the dimension $N$.
\end{lemma}%


\subsection{Finite difference tools} 

\label{discreate}



We will prove the theorem using Nirenberg's ``translation method." A good reference for this section is for instance Section 5.8.2. of \cite{Evans} but for the convenience of the reader, and since our statement are not exactly the same, we write the proofs with all details. For $h \in \R^N$,  and $\varphi : \R^N \to \R$, we denote 
$$\tau_{h}(\varphi)(x)=\varphi(x+h)$$
and then
$$D_h(\varphi) = \frac{\varphi(x+h) - \varphi(x)}{|h|} = \frac{\tau_{h}(\varphi)(x) - \varphi(x)}{|h|}.$$

The main ingredient is that a uniform $L^2$-control on $D_h(\varphi)$ (with respect to $h$),  leads to the conclusion that $\nabla \varphi \in L^2$, as stated in the following proposition. 
\begin{proposition} \label{Dh}
Let $\omega \subset \Omega \subset \R^N$ be two open sets such that  $\overline{\omega} \subset \Omega$ and let $u \in L^2(\Omega)$.  Then the following two properties holds:
\begin{enumerate}
\item  If $u \in H^1(\Omega)$ then for every $|\varepsilon | < \text{dist}(\omega, \Omega^c)$,
$$\|D_{\varepsilon e_i}(u)\|_{L^2(\omega)} \leq \| \partial_i u\|_{L^2(\Omega)}.$$
\item Conversely, if  $u \in L^2(\Omega)$ and  if there exists $C_0 > 0$ such that all $|\varepsilon | < \text{dist}(\omega, \Omega^c)$, it holds $\|D_{\varepsilon e_i}(u)\|_{L^2(\omega)} \leq C_0,$ then $\partial_i u \in L^2(\omega)$ and we have
\begin{eqnarray}
 \| \partial_i u\|_{L^2(\omega)} \leq C_0. \label{estimateGrad}
\end{eqnarray}
\end{enumerate}
\end{proposition}

\begin{proof}
We start by assuming that $u \in C^\infty(\Omega) \cap H^1(\Omega)$. Let $\varepsilon $ be such that $|\varepsilon| \leq \text{dist}(\omega, \Omega^c)$. We write
$$u(x+\varepsilon e_i) - u(x) = \int_{0}^1 \partial_i u(x+t \varepsilon e_i)  \, dt.$$
Using the Cauchy-Schwarz inequality and Fubini's theorem, it follows that
$$\int_{\omega} |u(x+\varepsilon e_i) - u(x)|^2 \, dx \leq \varepsilon^2 \int_0^1 \int_{\omega} |\partial_i u(x+t\varepsilon e_i)|^2 \, dx \, dt.$$
By the change of variables $y = x+t \varepsilon e_i$, we find
$$\int_{\omega} |u(x+\varepsilon e_i) - u(x)|^2 \, dx \leq |\varepsilon|^2 \int_0^1 \int_{\omega+t \varepsilon e_i} |\partial_i u(y)|^2 \, dy \, dt.$$
If $|\varepsilon| \leq \text{dist}(\omega, \Omega^c)$, then $\omega+t \varepsilon e_i \subset \Omega$, so finally
$$\|D_{\varepsilon e_i}(u)\|_{L^2(\omega)} \leq \| \partial_i u\|_{L^2(\Omega)}.$$
Now if $u \in H^1(\Omega)$, the result follows by approximation with $C^\infty$ functions, which concludes the proof of the first assertion.

Let us now prove the converse. We take   $0 < \varepsilon < \text{dist}(\omega, \Omega^c)$ and we define $g_\varepsilon := D_{\varepsilon e_i} u$, so that the hypothesis shows that the family 
$(g_\varepsilon)_{\varepsilon > 0}$ is bounded in $L^2(\omega)$. Therefore, there exists a subsequence such that $g_\varepsilon \to g^{(i)}$ weakly in $L^2(\omega)$. In particular
$$\|g^{(i)}\|_{L^2(\omega)} \leq \liminf \|D_{\varepsilon e_i} u\|_{L^2(\omega)} \leq C.$$

Let $\varphi \in C^\infty_c(\omega)$. For $\varepsilon < \text{dist}(\omega, \Omega^c)$, we have
\begin{eqnarray}
\int_{\Omega} (D_{\varepsilon e_i} u(x)) \varphi(x) dx &=& \int_{\Omega} \frac{u(x+\varepsilon e_i) - u(x)}{\varepsilon} \varphi(x) dx \notag \\
&=& -\int_{\Omega} u(y) \frac{\varphi(y) - \varphi(y-\varepsilon e_i)}{\varepsilon} dy \notag \\
&=& -\int_{\Omega} u(y) D_{-\varepsilon e_i} \varphi(y) dy. \notag
\end{eqnarray}
Moreover, since $D_{-\varepsilon e_i} \varphi(y) \to \partial_i \varphi(y)$ for all $y$, by the dominated convergence theorem, we deduce that
$$\int_{\omega} g^{(i)} \varphi dx = -\int_{\omega} u \partial_i \varphi dx,$$
which shows that $\partial_i u \in L^2(\omega)$ with $\partial_i u = g^{(i)}$ and 
$$\int_{\omega} |\partial_i u|^2 dx \leq C_0,$$
where $C_0 > 0$ is the constant from the statement. 
\end{proof}

By iterating the above proposition, we can prove the same for higher order derivatives. More precisely, we consider a vector of direction $(h_i)_{1\leq i\leq N}$ with $h_i \in \R^N$  and the associated  discretized operator $D_{h_{i}}$ defined by
    \begin{align*}
    D_{h_{i}}\varphi:= \frac{\tau_{h_{i}}\varphi -\varphi}{ \vert h_i \vert}.
    \end{align*}
    These operators commute two by two ; for every $i,j \in \llbracket 1,N \rrbracket$, $\vert h_i \vert = \vert h_j \vert$, one has
    \begin{align*}
      D_{h_{i}} D_{h_{j}}\varphi = \frac{1}{\vert h_i \vert}\tau_{h_{i}}\left(\frac{\tau_{h_{j}}\varphi -\varphi}{ \vert h_j \vert} \right)- \frac{\tau_{h_{j}}\varphi -\varphi}{ \vert h_i \vert \vert h_j \vert} = \frac{\tau_{h_{i}}\tau_{h_{j}}\varphi - \tau_{h_{i}}\varphi -\tau_{h_{j}}\varphi - \varphi}{\vert h_i \vert^2},
    \end{align*}
   but $\tau_{h_{i}}\tau_{h_{j}} = \tau_{h_{j}}\tau_{h_{i}}$, thus $   D_{h_{i}} D_{h_{j}} =  D_{h_{j}}  D_{h_{i}}$.
    Futhermore, they commute with the operators $\partial^\beta$, namely $D_{h_{i}}\partial^\beta \varphi = \partial^\beta D_{h_{i}}\varphi$. 
    
Then for  $\alpha \in \N^N$  and $\varepsilon \in \R$ we define
    \begin{align}
   D_{\varepsilon }^\alpha := \prod_{\substack{i=1\\ \text{ s.t. }\alpha_i\not = 0}}^{N} \prod_{l=1}^{\vert \alpha_i \vert}  D_{\varepsilon  e_i}. \label{defDalphah}
    \end{align}

   Notice that by a simple change of variables,     for every $u, v\in L^2(\R^N)$ we have 
    \begin{align*}
        \int_{\R^N}  u D_{\varepsilon}v\;dx = -\int_{\R^N} vD_{-\varepsilon}u\;dx,
    \end{align*}
and iterating this property we obtain
\begin{align*}
       \int_{\R^N} u D_{\varepsilon}^\alpha v\;dx =  - \int_{\R^N} v D_{-\varepsilon}^\alpha u\;dx.
\end{align*}


By applying $m$-times Proposition \ref{Dh}, we arrive to the  following direct corollary.

\begin{proposition} \label{Dhh}
Let $\omega \subset \Omega \subset \R^N$ be two open sets such that  $\overline{\omega} \subset \Omega$ and let $u \in L^2(\Omega)$.  Then the following two properties holds:
\begin{enumerate}
\item  If $u \in H^m(\Omega)$ then for every $|\varepsilon | < \text{dist}(\omega, \Omega^c)/m$ and $|\alpha|=m$,
$$\|D_\varepsilon^\alpha(u)\|_{L^2(\omega)} \leq \| \partial^\alpha u\|_{L^2(\Omega)}.$$
\item Conversely, if  $u \in L^2(\Omega)$, and if  there exists $\varepsilon_0,C_0 > 0$ such that for  $|\alpha|=m$, and for all $|\varepsilon| \leq \varepsilon_0<\text{dist}(\omega, \Omega^c)/m$, one has 
$$\|D_\varepsilon^\alpha(u)\|_{L^2(\omega)} \leq C_0,$$
 then $\partial^\alpha u \in L^2(\omega)$ and we have
\begin{eqnarray}
 \| \partial^\alpha u\|_{L^2(\omega)} \leq C_0. \label{estimateGrad2}
\end{eqnarray}
\end{enumerate}
\end{proposition}

\begin{proof} We argue by induction: if $|\alpha|=1$ then the statement is exactly Proposition~\ref{Dh}.

So we assume that the statement is true for all $|\alpha|=m$ and we consider some $\alpha$ such that $|\alpha|=m+1$, and $\varepsilon< \text{dist}(\omega, \Omega^c)/(m+1)$. Then we can write 
$$D_\varepsilon^\alpha=D_{\varepsilon e_i} D_\varepsilon^\beta,$$
with  $|\beta|=m$ and for some $e_i$ arbitrary chosen.

For the first part of the statement, we assume that $u\in H^{m+1}(\Omega)$. Let $\omega'$ be an open set such that 
$\omega \subset \omega' \subset \Omega$, and moreover such that 
$$|\varepsilon|\leq  \text{dist}(\omega, (\omega')^c)/m \text{ and } |\varepsilon|\leq  \text{dist}(\omega', \Omega^c).$$
This is possible because $|\varepsilon | < \text{dist}(\omega, \Omega^c)/(m+1)$. Then the induction hypothesis says that  
$$ \| D_\varepsilon^\alpha(u)\|_{L^2(\omega)}=  \| D_\varepsilon^\beta( D_{\varepsilon e_i} u)\|_{L^2(\omega)} \leq \| \partial^\beta D_{\varepsilon e_i} u \|_{L^2(\omega')}= \| D_{\varepsilon e_i} \partial^\beta  u \|_{L^2(\omega')}.$$
But since $\partial^\beta  u \in H^1(\Omega)$, we can apply Proposition~\ref{Dh} which says that 
$$\| D_{\varepsilon e_i} \partial^\beta  u \|_{L^2(\omega')}\leq \| \partial_i \partial^\beta  u \|_{L^2(\Omega)}=\|   \partial^\alpha  u \|_{L^2(\Omega)},$$
which  proves the first part of the proposition.

Now for the second part, we assume that $u \in L^2(\Omega)$, and  there exists $C_0 > 0$ such that for every $|\alpha|=m+1$, and for all $|\varepsilon| \leq \varepsilon_0\leq  \text{dist}(\omega, \Omega^c)/m$, one has 
$$\|D_\varepsilon^\alpha(u)\|_{L^2(\omega)} \leq C_0.$$
Then again we write   
$$D_\varepsilon^\alpha=D_{\varepsilon e_i} D_\varepsilon^\beta,$$
with  $|\beta|=m$ and for some $e_i$ arbitrary chosen.  By applying the induction hypothesis we know that $\partial^\beta (D_{\varepsilon e_i}u)\in L^2(\omega)$ and 
$$ \| \partial^\beta (D_{\varepsilon e_i}u)\|_{L^2(\omega)} \leq C_0. $$
This actually yields,
$$ \|D_{\varepsilon e_i}  \partial^\beta u\|_{L^2(\omega)} \leq C_0,$$
so a last application of Proposition~\ref{Dh} finally gives 
$$  \|  \partial^\alpha u\|_{L^2(\omega)}=\|\partial_i \partial^\beta u\|_{L^2(\omega)} \leq C_0,$$
and so follows  the proposition.
\end{proof}

\subsection{Campanato spaces}

We recall here some standard facts about Campanato spaces that will be used to achieve our $C^{m-1,\alpha}$ regularity result.

\begin{definition}
   Let $\Omega$ be an open set of $\R^N$. The Campanato space of parameter $\lambda\in \R^+$ is 
    \begin{align*}
        \mathcal{L}^{2, \lambda}(\Omega) := \left\{ u\in L^p(\Omega) \; \middle| \; [u]^2_{\lambda} := \sup_{\substack{x \in \Omega \\ 0 < r < r_0}} r^{-\lambda}\int_{\Omega \cap B(x,r)}\vert u -    m(u;x,r) \vert^2 dx < \infty \right\},
    \end{align*}
    where
    \begin{align*}
        m(u;x,r) := \frac{1}{\vert \Omega \cap B(x,r) \vert}\int_{\Omega \cap B(x,r)}u \; dx.
    \end{align*}
\end{definition}

  \begin{proposition}
    Let $\Omega$ be an $(\varepsilon_0,r_0)-$Reifenberg flat domain in $\R^N$. If $\lambda \in( N, N+2]$, then there exists a continuous injection
    \begin{align}\label{campanato}
        \mathcal{L}^{2,\lambda}(\Omega) \xhookrightarrow{\quad } \mathscr{C}^{0,\alpha}(\overline{\Omega}),
    \end{align}
    with $ \alpha :=(\lambda - N )/p$.
  \end{proposition}

 A proof can be found in \cite[Theorem 3.1]{Gi93} where  bounded open domains with following the density condition are considered:  there exists $A>0$ such that 
\begin{align*}
    \vert \Omega \cap B(x,r) \vert \geq A r^N, \quad 0 < r \leq r_0,\quad x\in \overline{\Omega}. \newline
\end{align*}
In our case, this is true since Reifefnberg-flat domain satisfy internal corkscrew condition (see for instance \cite{GLM23}).

  \begin{proposition}\label{continuitedecroissance}
    Let $\Omega$  be a $(\varepsilon_0,r_0)-$Reifenberg flat domain and $u \in H^m_0(\Omega)$. If $\alpha \in (0,1]$ is such that for all $x\in \Omega$ and $r \in(0,r_0)$,
    \begin{align*}
        \int_{B(x,r)}\vert \nabla^m u(y) \vert^2 dy \leq C_0 r^{N+2\alpha-2},
    \end{align*}
   then $u\in \mathscr{C}^{m-1,\alpha}(\overline{\Omega})$ and 
   $$\|u\|_{\mathscr{C}^{m-1,\alpha}(\overline{\Omega})} \leq CC_0,$$
   where $C$ depends on $N$ and on ${\rm diam}(\Omega)$.
  \end{proposition}
\begin{proof}
 Let $v=\partial^{\gamma}u$ where $|\gamma|=m-1$.  Using the Poincar\'e inequality, we get
    \begin{align*}
        \int_{B(x,r)}\left\vert v(y) - m(v;x,r) \right\vert^2 dy &\leq \int_{B(x,r)}\left\vert v(y) - m(v;x,r) \right\vert^2 dy \\
        &\leq C(N)r^2  \int_{B(x,r)}\left\vert \nabla v(y) \right\vert^2 dy \\
            &\leq C(N) C_0 r^{N+2\alpha }.
    \end{align*}
    Hence, $v\in \mathcal{L}^{2, N+ 2\alpha }(\Omega)$ and with the injection \eqref{campanato}, we know that $v\in C^{0,\alpha}(\overline{\Omega})$. But then  for all $x\in \Omega$ and $r\leq r_0$,
        \begin{align*}
        \int_{B(x,r)}\vert  v(y) \vert^2 dy \leq C_0 r^{N+2\alpha},
    \end{align*}    
    or in other words,
           \begin{align*}
        \int_{B(x,r)}\vert  \nabla^{m-1}v(y) \vert^2 dy \leq C_0 r^{N+2\alpha},
    \end{align*}    
and from the same argument as above we deduce that $\partial^\gamma v \in C^{0,\alpha}(\overline{\Omega})$ for all $|\gamma|=m-2$. We can continue like this successively up to arrive to  fact that  $u\in \mathscr{C}^{m-1,\alpha}(\overline{\Omega})$.
\end{proof}

\section{Interior regularity by the   translation method}
 
We start with the interior regularity, which is easier  since the translations in all directions are admissible as test functions. This part works in a similar way  as in the usual standard case of the Laplacian, but with some difficulty of working with higher order derivatives.   For the convenience of the reader, and since it will be generalized later in a non-trivial way at the boundary, we write the full details.

 \begin{theorem}[Interior regularity]\label{theoremederegularite}
    Let $m\in \N^*$. Let $\mathscr{A}$ be an   operator of order $2m$ of the form \eqref{ellipticOperator}, and satisfying the ellipticity condition \eqref{ellipticPartial}. Let  $u \in H^m(B(0,1))$ be a weak solution for $ \mathscr{A}u=f$ in $B(0,1)$ with  $f \in H^l(B(0,1))$. Then $u \in H^{2m+l}(B(0,1/2))$ and
    $$\|u\|_{H^{2m+l}(B(0,\frac{1}{2}))}\leq C(\|u\|_{H^{m}(B(0,1))} + \|f\|_{H^l(B(0,1)}),$$
 where $C>0$ is a constant that depend  on    $N$,  $\max_{\alpha,\beta} |a_{\alpha,\beta}|$, $l$ and $m$.
  \end{theorem}
 
\begin{proof} We will denote by $B=B(0,1)$. It is enough to prove the theorem for $f\in L^2(B)$. Indeed, assuming  that it holds true in this case. Then for $f\in H^l(B)$ one can apply the statement to the function $\partial^\alpha u$ with $|\alpha|\leq l$ that satisfies, in the sense of distributions,
\begin{eqnarray}
 \mathscr{A}\partial^\alpha u=\partial^\alpha f, \label{equationD}
 \end{eqnarray}
with $\partial^\alpha f \in L^2$. This would imply $\partial^\alpha u \in H^{2m}$ with
$$\| \partial^\alpha u\|_{H^{2m}(B(0,\frac{1}{2}))}\leq C(\| \partial^\alpha u\|_{H^{m}(B)} + \|\partial^\alpha f\|_{L^2(B)}),$$
thus we could easily conclude by an inductive argument on $|\alpha|=0,1,\dots,l$. The only thing is that since we would need to apply the result $l$ times, the conclusion would hold in the smaller ball $B(0,\frac{1}{2^l})$ instead of $B(0,1/2)$. But the constant are here arbitrary and the same result could be done from $B(0,R)$ to $B(0,r)$ for any $r<R$ without any substantial modification.

Therefore, in the sequel we will assume that $f\in L^2(B)$. The main starting point is  the weak formulation of the problem which is 
\begin{align*}
         \int_{B}A\nabla^m u: \nabla^m \varphi \;dx = \int_{B}f\varphi \; dx, \quad \textnormal{for every } \varphi \in H^m_0(B).
     \end{align*}


In the sequel we will use an inductive argument, assuming that  $u\in H^{m+k}(B_k)$. We want to prove that $u\in H^{m+k+1}(B_{k+1})$ for a slightly smaller ball $B_{k+1}\subset B_k$, up to arrive at $B(0,1/2)$. For that purpose we define, for  $1\leq k\leq m$, the ball $B_k := B(0,1/2+(1/2)^k)$. We will also assume without mentioning it explicitly that $|h|$ is so small that Proposition \ref{Dhh} always applies with $\omega= B_{k+1}$ and $\Omega=B_k$.

Next, we fix  $|\gamma|=k\leq m$.   We will use the  discrete derivatives defined in Section \ref{discreate}. The method of  translations is based on the fact that for all $h\in \R^N$ such that $|h|$ is small enough, up to multiply by a smooth cut-off function, the function $D_h u$ is admissible as a test function in the weak formulation, as being in $H^{m}$.  This will provide an estimate on $D_h u$ that will allows us to win a derivative. Since we need to win $m$ derivatives will apply this $m$ times.

The main ingredient is  the weak formulation and the change of variables formula that yields
\begin{align}
         \int_{B_k}A\nabla^m D_h(D_{\varepsilon}^{\gamma}u) \cdot \nabla^m \varphi \;dx =
          (-1)^{|\gamma|+1}\int_{B_k} f  D_{-h}D_{-\varepsilon}^\gamma\varphi \; dx, \quad \textnormal{for every } \varphi \in H^m_0(B_k). \label{weakStart}
     \end{align}
  To lighten the notation, from now on we will denote by $v=D_\varepsilon^\gamma u$.

     Now  we consider    a cut-off function $\eta \in \mathscr{C}_c^\infty(B_k)$ such that for every $x\in \R^N$, 
\begin{align*}
    \eta(x) \in [0,1], \quad \eta(x) = 1 \textnormal{ on } B_{k+1},\quad \textnormal{and } \vert \nabla^j \eta(x)\vert \leq C, 
\end{align*}
where $C>0$, could depend on $k$. Then we take $\varphi=\eta^2 D_h D^\gamma_{\varepsilon} u=\eta^2 D_h  v$ as a test function, for $|h|\leq 1/4$,
 which is admissible in the space $H^{m}_0(B_k)$.

Using Proposition \ref{Dhh}, the term on the right-hand side of \eqref{weakStart} can be estimated as 
\begin{eqnarray}
 \left| \int_{B} f D_{-h}(D_{-\varepsilon}^\gamma  (\eta^2 D_h  v)) \; dx\right|\leq  C \|f\|_{L^2(B_k)}\|\nabla   \partial^\gamma (\eta^2 D_h  v)\|_{L^2(B_k)}. \label{estimation000}
 \end{eqnarray}

We will use this inequality later. For now on we develop   the term $\nabla^m (\eta^{2}D_{h}v)$: for every $\alpha \in \N^N$ such that $\vert \alpha \vert =m$, the Leibniz formula (see \eqref{leibniz}) says 
\begin{eqnarray}
    \partial^\alpha (\eta^{2}D_{h}v) &=& \sum_{\beta\leq \alpha}\frac{\alpha!}{\beta! (\beta-\alpha)!}\partial^{\beta}(\eta^{2})\partial^{\alpha-\beta}(D_{h}v), \notag \\
    &=& \eta^{2}\partial^{\alpha}(D_{h}v) + \sum_{0<\beta\leq  \alpha}\frac{\alpha!}{\beta! (\beta-\alpha)!}\partial^{\beta}(\eta^{2})\partial^{\alpha-\beta}(D_{h}v).  \notag 
    \end{eqnarray}

In other words we have, in terms of vectors, 
\begin{eqnarray}
\nabla^m(\eta^2 D_h v) = \eta^2 \nabla^m D_hv + L(\eta,v), \label{decomposition}
\end{eqnarray}
where 
$$L(\eta,v)_\alpha:= \sum_{0<\beta\leq  \alpha}\frac{\alpha!}{\beta! (\beta-\alpha)!}\partial^{\beta}(\eta^{2})\partial^{\alpha-\beta}(D_{h}v).$$
 
Notice that the expression in $L(\eta,v)_\alpha$ contains  derivatives in $v$ or order  at most $m$, including the extra derivative in the direction of $x_i$ provided by $D_h$.   Indeed, for the particular case $\alpha= me_i$ (i.e. when $\partial^\alpha=\partial_{x_i}^{(m)}$) then the condition $0<\beta\leq \alpha$ means that $\beta_i>0$ so that $|\alpha-\beta+e_i|\leq m$, as claimed. On the other hand for $\alpha \not = me_i$ since $|\alpha|=m$, one must have  $\alpha_i<m$ so that, again, $|\alpha-\beta+e_i|\leq m$.

Notice also that since $\beta>0$, $\partial^{\beta}(\eta^{2})$ always contains at least a factor $\eta$. We deduce that there exists an expression $M(\eta,u)$ such that 
\begin{eqnarray}
L(\eta,v)=\eta M(\eta,v), \label{decomposition2}
\end{eqnarray}
 and $M(\eta,v)$ satisfies,
\begin{eqnarray}
\| M(\eta,v)  \|_{L^2(B_k)}\leq C\|v\|_{H^m(B_k)}\leq C\|u\|_{H^{m+k}(B_k)}.
\end{eqnarray}

Next, we deduce from the ellipticity of $A$  that 
\begin{eqnarray} 
     \int_{B} A\nabla^m D_{h} v \cdot \nabla^m (\eta^{2} D_{h}v)\;dx &= &        \int_{B} A\nabla^m D_{h} v \cdot \eta^{2}  \nabla^m D_{h}v\;dx  +   \int_{B} \eta A\nabla^m D_{h} v \cdot  M(\eta,v)   \;dx \notag \\
     &\geq &    C\int_{B}  \eta^2 |\nabla^m D_{h} v|^2 \;dx  +   \int_{B} \eta A\nabla^m D_{h} v \cdot  M(\eta,v)   \;dx .\label{estimation220}
\end{eqnarray}
Moreover, 
\begin{eqnarray}
 \left| \int_{B} \eta A\nabla^m D_{h} v \cdot  M(\eta,v)   \;dx\right| &\leq&   \frac{\lambda}{2} \sup |A| \int_{B} \eta^2 |\nabla^m D_{h} v|^2 \;dx +  \frac{1}{2\lambda}  \int_{B} |M(\eta,v)|^2  \;dx \notag \\
 &\leq & \frac{1}{2}\int_{B} \eta^2 |\nabla^m D_{h} v|^2 \;dx  + C\|v\|_{H^m(B_k)}^2 \label{estimation230}
 \end{eqnarray}

Returning to \eqref{weakStart} and using  \eqref{estimation220}, \eqref{estimation230} and \eqref{estimation000}, we have finally obtained, using also Proposition \ref{Dhh} again, that 
\begin{eqnarray}
\int_{B} \eta^2 |\nabla^m D_{h} v|^2 \;dx  &\leq& C \|v\|^2_{H^m(B_k)}+ C \|f\|_{L^2(B)}\|\nabla   \partial^\gamma (\eta^2 D_h  v)\|_2 \notag \\
&\leq& C \|v\|^2_{H^m(B_k)}+ C \|f\|_{L^2(B)}\|v\|_{H^{k+2}(B_k)}.
\end{eqnarray}
Since $v=D^\gamma_{\varepsilon} u$, with $\gamma$ any multi-index with $|\gamma|=k$,  and since $\eta=1$ on $B_{k+1}$, the above inequality implies, thanks to Proposition \ref{Dhh},
\begin{eqnarray}
\int_{B_{k+1}}  |\nabla^{m+k+1} u|^2 \;dx   \leq C ( \|u\|^2_{H^{m+k}(B_k)}+  \|f\|_{L^2(B)}^2 + \|u\|_{H^{2k+2}(B_k)}^2). \notag 
\end{eqnarray}
 If $k\leq m-2$, This directly gives that $u\in H^{m+k+1}(B_{k+1})$ (because we assumed by induction that $u\in H^{m+k}(B_{k})$.

 For $k=m-1$ we have to argue a bit differently  using Young inequality. Indeed,  the same Leibniz formula as \eqref{decomposition} with $k+1$ in place of $m$ in that formula, yieds
 \begin{eqnarray}
\nabla^{k+1}(\eta^2 D_h v) = \eta^2 \nabla^{k+1} D_hv + L(\eta,v), \label{decomposition00}
\end{eqnarray}
 where 
 \begin{eqnarray}
\| L(\eta,v)  \|_{L^2(B)}\leq C\|v\|_{H^{k+1}(B_k)}\leq C\|u\|_{H^{2k+1}(B_k)}.
\end{eqnarray}
 Therefore, using Young's inequality we can estimate 
 \begin{eqnarray}
\|\nabla   \partial^\gamma (\eta^2 D_h  v)\|_2&\leq& \delta \|\eta \nabla^{2k+2}u\|_{2} + C(\delta)\|u\|_{H^{2k+1}(B_k)}. \notag 
\end{eqnarray}
Since $2k+2=m+k+1=2m$, all together we have proved, in the case when $k=m-1$, that 

\begin{eqnarray}
\int_{B_{k+1}}  |\nabla^{2m} u|^2 \;dx   \leq C ( \|u\|^2_{H^{2m-1}(B_k)}+  \|f\|_{L^2(B)}^2 ). \notag 
\end{eqnarray}

  In any case we have proved,  that $\nabla^{m+k+1} u \in L^2(B_{k+1})$ and 
$$\|\nabla^{m+k+1}u\|_{L^2(B_{k+1})}\leq C(\|u\|_{H^{m+k}(B_k)} + \|f\|_{L^2(B)}).$$
All in all,  we have proved that if $u \in H^{m+l}(B_k)$ then $u\in H^{m+k+1}(B_{k+1})$. We can then argue by induction, since for $k=0$ it is true that $u\in H^m(B)$, and since the induction works until $k=m-1$, we arrive at the conclusion that $u\in H^{2m}(B(0,1/2))$, and so follows the Theorem.
\end{proof}

We will need  later the following Corollary.

 \begin{corollary}[Interior regularity with $f=0$]\label{theoremederegularite3INT}
 Let $\mathscr{A}$ be an   operator of order $2m$ of the form \eqref{ellipticOperator}, and satisfying the ellipticity condition \eqref{ellipticPartial}.   Let  $u \in H^m(B(0,R))$ be a weak solution for $\mathscr{A} u=0$ in $B(0,R)$.  Then $\nabla^m u \in L^\infty(B(0,R/2))$ and
    \begin{eqnarray}
    \|\nabla^m u\|_{L^\infty(B(0,\frac{R}{2}))}\leq C_0\|\nabla^m u\|_{L^{2}(B(0,R))}, \label{estimateInt}
    \end{eqnarray}
 where $C_0>0$ is a constant that depends  on    $N$,  $\max_{\alpha,\beta} |a_{\alpha,\beta}|$,  and $m$. In particular there exists $R_0>0$ depending only on dimension $N$,  $\max_{\alpha,\beta} |a_{\alpha,\beta}|$,  and $m$, such that  for all $r\leq R\leq R_0$ we have
   \begin{align}\label{propdecroissanceenergiedebutformule0}
          \int_{B(x,r) }\vert \nabla^m u \vert^2 dx \leq  C_N \left(\frac{r}{R}\right)^{N} \int_{B(x,R)}\vert \nabla^m u \vert^2 dx,
    \end{align}
    where $C_N$ is a constant that depends  only on    $N$.
   \end{corollary}

\begin{proof}  Assume first that $R=1$. Then we can apply Theorem \ref{theoremederegularite} to obtain  that $\nabla^m u\in H^\ell (B^+(0,R/2))$  for all $\ell\in \N$ because here $f=0$. Then the Sobolev embedding theorem that says that, provided $2\ell>N$, then $H^\ell \hookrightarrow \mathscr{C}^{0,\alpha}$, and in particular $\nabla^m u\in L^\infty$ with

 $$\|\nabla^m u\|_{L^\infty(B(0,1/2))}\leq C\|u\|_{H^{m}(B(0,1))}.$$

We would like to apply this inequality to the rescaled function $u(Rx)$  but  the problem is that the norm above is not homogeneous. For that purpose we consider $u-m$ where $m$ is the average of $u$ on $B(0,1)$. Then we notice that $\nabla^m u=\nabla^m(u-m)$ and applying the above to $u-m$ instead of $u$ yields
 $$\|\nabla^m u\|_{L^\infty(B(0,1/2))}\leq C\sum_{k=1}^m\|\nabla^k(u-m)\|_{L^2(B(0,1))}.$$
By applying Poincar\'e inequality we actually get 
 $$\|\nabla^m u\|_{L^\infty(B(0,1/2))}\leq C\sum_{k=2}^m\|\nabla^k u\|_{L^2(B(0,1))}.$$
Then we apply the same argument to $u-m x_i$ where $m$ is the average of $\partial_i u$. Etc. After $m-1$ application of Poincar\'e inequality we finally obtain
$$\|\nabla^m u\|_{L^\infty(B(0,1/2))}\leq C_0 \|\nabla^m u\|_{L^2(B(0,1))}.$$
The above inequality  is now homogenous after scaling, thus applying it to the rescaled function $u(Rx)$ we finally arrive to \eqref{estimateInt}.

Let us now prove \eqref{propdecroissanceenergiedebutformule0}. For that we write the following chain of inequalities
  \begin{align} 
          \int_{B(x,R/2) }\vert \nabla^m u \vert^2 dx &\leq \|\nabla^m u \|_{L^\infty(B(x,\frac{R}{2}))} \omega_N \left(\frac{R}{2}\right)^N \notag \\ 
          &\leq   C_0 \omega_N \left(\frac{R}{2}\right)^{N} \int_{B(x,R)}\vert \nabla^m u \vert^2 dx, \notag \\
          &\leq \frac{1}{2^N} \int_{B(x,R)}\vert \nabla^m u \vert^2 dx,
    \end{align}
provided that $R_0^N  C_0 \omega_N \leq 1$. So that by iteration,
    \begin{align*} 
              \int_{B(x,2^{-k}R) }\vert \nabla^m u \vert^2 dx   \leq   \frac{1}{2^{kN}}  \int_{B(x,R)}\vert \nabla^m u \vert^2 dx.
    \end{align*}  
    Finally for $2^{-k+1}R \leq r \leq 2^{-k}R$,
      \begin{align*} 
          \int_{B(x,r) }\vert \nabla^m u \vert^2 dx  & \leq   \int_{B(x,2^{-k}r) }\vert \nabla^m u \vert^2 dx   \\
          &\leq  \frac{1}{2^{kN}}   \int_{B(x,R)}\vert \nabla^m u \vert^2 dx \\
          &\leq \frac{1}{2^N}  \left(\frac{r}{R}\right)^{N} \int_{B(x,R)}\vert \nabla^m u \vert^2 dx,
    \end{align*}  
    
which ends the proof.
\end{proof}
 \section{Boundary regularity in the case of a hyperplane}

\label{boundaryReg}

In this section we investigate the regularity up to the  boundary  for the Dirichlet problem, in the special case of a flat boundary. As before we consider an operator of order $2m$ of the form
\begin{align*}
        \mathscr{A} := (-1)^m\sum_{|\beta|=\vert \alpha \vert = m} a_{\alpha,\beta} \partial^\alpha \partial^\beta=(-1)^m A\nabla^m \cdot \nabla^m,
    \end{align*}
 and we focus this time on a weak solution $u\in H^m(B)$ for the problem in a half-ball:
 \begin{eqnarray}
\left\{ 
 \begin{array}{l}
         \mathscr{A}u= f \quad \text{ in } B^+(0,1):=B(0,1)\cap \{x_N>0\} \\
         u\in H^m(B(0,1)) \text{ and } u=0 \text{ a.e. in }  B(0,1)\cap \{x_N<0\}.
\end{array}
\right. \label{boundaryP}
\end{eqnarray}
The weak formulation for  \eqref{boundaryP} is 
 \begin{align*}
         \int_{B}A\nabla^m u: \nabla^m \varphi \;dx = \int_{B}f\varphi \; dx, \quad \textnormal{for every } \varphi \in H^m_0(B^+).
     \end{align*}

Now here below  is our regularity result at the boundary analogous to the interior estimate (Theorem \ref{theoremederegularite}). Notice that, compared to the interior estimate, below here  in \eqref{estimationB} the homogenous norm $\|\nabla^m u\|_2$ appears instead of the non homogeneous norm $\|u\|_{H^m}$. This is possible thanks to the Poincaré inequality available since a  Dirichlet  boundary condition is imposed on $u$ on the flat part of $B^+(0,1)$.

 \begin{theorem}[Regularity with  flat boundary]\label{theoremederegularite2}
    Let $m\in \N^*$. Let $\mathscr{A}$ be an  operator of order $2m$ of the form \eqref{ellipticOperator} and satisfying the ellipticity condition \eqref{ellipticPartial}. Let  $u \in H^m(B(0,1))$ a weak solution for $ \mathscr{A}u=f$ in $B^+(0,1)$ with  $f \in H^\ell(B^+(0,1))$, and such that $u=0$ in $B(0,1)\cap \{x_N<0\}$. Then $u \in H^{2m+\ell}(B^+(0,1/2))$ and
    \begin{eqnarray}
    \|u\|_{H^{2m+\ell}(B^+(0,\frac{1}{2}))}\leq C(\|\nabla^m u\|_{L^{2}(B^+(0,1))} + \|f\|_{H^\ell(B^+(0,1)}), \label{estimationB}
    \end{eqnarray}
 where $C>0$ is a constant that depend  on    $N$,  $\max_{\alpha,\beta} |a_{\alpha,\beta}|$, $l$ and $m$.
  \end{theorem}
 
 \begin{proof} The starting point of the proof is similar to  the interior case, but there are some differences. 
 
 The   main difference is that the translated functions $D_h u$ are admissible as a test function in $B^{+}(0,1)$ only for horizontal directions, i.e. for $h\in \R e_i$ and $1\leq i\leq N-1$. Indeed,  for $i=N$ the vertical translation would not preserve the boundary condition.


So let us assume first that  $f\in L^2(B^+(0,1))$ (the general case  $f\in H^\ell(B^+(0,1))$ will be treated later). For  $\gamma\in \N^N$ containing only horizontal directions, in other words  such that $\gamma_N=0$, we can argue exactly as in the case of the interior regularity. We do not write again the details, and refer to the proof of Theorem \ref{theoremederegularite}. This reasoning yields that  $\partial^\gamma \nabla^m u \in L^2(B^+(0,1/2))$ for all $|\gamma|\leq m$ such that  $\gamma_N=0$, with the estimate
\begin{eqnarray}\label{estimHor}
 \|\partial^\gamma \nabla^m u\|_{L^2(B^+(0,\frac{1}{2}))}\leq C(\|\nabla^m u\|_{L^{2}(B^+(0,1))} + \|f\|_{L^2(B^+(0,1)}),
 \end{eqnarray}
 where $C>0$ is a constant that depends  on    $N$,  $\max_{\alpha,\beta} |a_{\alpha,\beta}|$, and $m$. Actually,   the proof of Theorem \ref{theoremederegularite} would give $\|u\|_{H^m(B^+(0,1))}$ on the right-hand side but thanks to the Poincar\'e inequality (Lemma \ref{poincare}) and using  that $u=0$  on $B\cap \{x_N<0\}$ we can bound 
 $$\|u\|_{H^m(B^+(0,1))} \leq C\|\nabla^m u\|_{L^{2}(B^+(0,1))},$$
 from which we get \eqref{estimHor}.

Now to prove the theorem we want to recover all  the derivatives in the vertical direction, up to $2m$-order. For this purpose we will argue by induction.

\quad \textit{Step 1.} We suppose that $m=1$, then  we just proved that $\partial_i\partial_j u \in L^2(B^+(0,1/2))$ for all $i \in \llbracket 1, N-1 \rrbracket$ and $j \in \llbracket 1, N \rrbracket$. The operator $\mathscr{A}$ is of the form
\begin{align*}
    \mathscr{A}= -\sum_{i=1}^N \sum_{j=1}^N a_{i,j}\partial_i \partial_j,
\end{align*}
and by ellipticity we know that $a_{N,N}\not = 0$. From the identity $ \mathscr{A}u = f \in L^2(B^+(0,1))$ then
\begin{align*}
  a_{N,N} \partial_N^2u =  -f + \sum_{(i,j)\not = (N,N)}a_{i,j}\partial_i \partial_j u \in L^2(B^+(0,1/2)),
\end{align*}
and we conclude that $u \in H^2(B^+(0,1/2))$ with
 $$\|u\|_{H^{2}(B^+(0,\frac{1}{2}))}\leq C(\|\nabla u\|_{L^{2}(B^+(0,1))} + \|f\|_{L^2(B^+(0,1)}).$$

 This proves the desired conclusion for the particular  case when $m=1$.
 
 \medskip

\quad  \textit{Step 2.} We assume that $m\geq 2$  and we suppose that the theorem is true for operators of order $2(m-1)$. Let us prove that it is still true for operator of order $2m$. For that purpose we decompose $\mathscr{A}$ as in the statement of Lemma \ref{decomposition0} :
\begin{eqnarray}
 \mathscr{A}&=&(-1)^m\sum_{|\alpha|=\vert \beta \vert = m} a_{\alpha,\beta}\partial^\beta \partial^\alpha\notag \\
 &=&(-1)^m\sum_{\substack{|\alpha|=\vert \beta \vert  = m \\ \alpha_N=0 \text{ or } \beta_N=0}} a_{\alpha,\beta} \partial^\beta\partial^\alpha+(-1)^m\sum_{\substack{|\alpha|=\vert \beta \vert = m \\ \alpha_N>0 \text{ and  } \beta_N>0}}  a_{\alpha,\beta} \partial^\beta \partial^\alpha. \notag 
 \end{eqnarray}
In other words $  \mathscr{A}=      \mathscr{B} +    \mathscr{C}$ with

$$\mathscr{B}:= (-1)^m\sum_{\substack{|\alpha|=\vert \beta \vert = m \\ \alpha_N=0 \text{ or } \beta_N=0}} a_{\alpha,\beta} \partial^\beta\partial^\alpha$$

and 

$$ \mathscr{C}:= (-1)^m\sum_{\substack{|\alpha|=\vert \beta \vert = m \\ \alpha_N>0 \text{ and  } \beta_N>0}} a_{\alpha,\beta} \partial^\beta\partial^\alpha.$$

From the first part of the proof, we already know that $\mathscr{B}(u)\in L^2(B^+(0,1/2))$ because  it contains at most $m$ derivatives in the vertical direction.

Now we consider the function  $w:=-\partial_N^2 u$ and we notice that  $\mathscr{C}$ is an operator that acts as an operator of  $2m-2$ order on the function $w$.  Indeed,
$$ \mathscr{C}(u)= (-1)^m\sum_{\substack{|\alpha|=\vert \beta \vert = m \\ \alpha_N>0 \text{ and  } \beta_N>0}} a_{\alpha,\beta} \partial^\beta\partial^\alpha u=(-1)^m \sum_{|\alpha'|=\vert \beta' \vert = m-1 } a_{\alpha'+e_N,\beta'+e_N} \partial^\beta \partial^\alpha  \partial^{2}_N u=\mathscr{D}(w),$$
where $\mathscr{D}$ is an operator of order $2m-2$. Furthermore, thanks to Lemma \ref{decomposition}, we know that   $\mathscr{D}$ is elliptic.

 Since $u= H^m(B(0,1))$, it follows that $w\in H^{m-2}(B(0,1))$. Moreover, since $u=0$ on the open set $B(0,1)\cap \{x_N<0\}$, the same still holds true for $w$.     Then we can write
 $$\mathscr{D}(w)=\mathscr{A}(u)- \mathscr{C}(u)= f - \mathscr{C}(u) \in L^2(B(0,1)).$$
Since $\mathscr{D}$ is of order $2m-2$, we would like to deduce from the induction hypothesis that $w \in H^{2m-2}(B^+(0,1))$. But unfortunately there is a gap of regularity here, because we only know that $w\in H^{m-2}(B(0,1))$ and to use the induction hypothesis we would need that  $w\in H^{m-1}(B(0,1))$.

To overcome this  problem we will use the Poincar\'e inequality in Lemma \ref{inegalitedeplacementmasse}. Indeed, we can  rewrite $\mathscr{D}(w)$ as follows:
$$\mathscr{D}(w)= -\partial_N \mathscr{D}(\partial_N u).  $$
Since $\mathscr{D}(w) \in L^2(B^+(0,1/2))$, we deduce that the distribution $z:= \mathscr{D}(\partial_N u)$ verifies that 
$$\partial_N z \in L^2(B^+(0,1/2)).$$
 Moreover, thanks to the interior regularity (Theorem \ref{theoremederegularite}), we know that 
 $$z\in L^{2}(B^{+}(0,1/2)\cap \{dist(x,\partial B^+(0,1/2))>\delta_0\}).$$
  By   applying  Lemma \ref{inegalitedeplacementmasse} we deduce that $z\in L^2(B^+(0,1/2))$.  In other words $\mathscr{D}(\partial_N u)\in L^{2}(B^{+}(0,1/2)$ and  now  since $\partial_N u\in H^{m-1}(B)$, with $\partial_N u =0$ on $B\cap \{x_N<0\}$, the induction hypothesis applies. We conclude that  $\partial_N u \in H^{2m-2}(B^+(0,1/4))$. This fact, together with the previous analysis leads to say that $u \in H^{2m-1}(B^+(0,1/4))$ and
 $$\|u\|_{H^{2m-1}(B^+(0,\frac{1}{4}))}\leq C(\|\nabla^m u\|_{L^{2}(B^+(0,1))} + \|f\|_{L^2(B^+(0,1)}).$$
Now we  are still missing one derivative because we need to achieve $H^{2m}$ instead of $H^{2m-1}$. This will be done by use of the following observation: all in all, at this stage of the proof, we have a control on  all the derivatives of the form $\partial^{\alpha}u$ with $|\alpha|=2m$, provided that  $|\alpha_N|\leq 2m-1$. In other words, to conclude that $u \in H^{2m}$ we only miss the last derivative $\partial^{2m}_N u$.  But this term appears in  $\mathscr{A}$ only once. Namely, we have 
\begin{eqnarray}
\mathscr{A}=(-1)^m a_{me_N, me_N} \partial^{2m}_N  + \mathscr{E}  \label{decompNN}
\end{eqnarray} 
where
$$\mathscr{E}:=(-1)^m\sum_{\substack{|\alpha|=\vert \beta \vert = m \\ \alpha_N<m \text{ or } \beta_N<m}} a_{\alpha,\beta} \partial^\beta\partial^\alpha.$$
Since $\mathscr{E}$ contains at most $2m-1$ derivatives in the $e_N$ direction, we know that $\mathscr{E}(u)\in L^2(B^+(0,1/4))$. Moreover, by ellipticity of $\mathscr{A}$ we deduce that the coefficient $a_{me_N, me_N}\not =0$, and therefore, from \eqref{decompNN} we get  $\partial^{2m}_N u \in L^2(B^+(0,1/4))$. In conclusion we have obtained,
 $$\|u\|_{H^{2m}(B^+(0,\frac{1}{4}))}\leq C(\|\nabla^m u\|_{L^{2}(B^+(0,1))} + \|f\|_{L^2(B^+(0,1)}).$$
So follows the claim, with a radius $1/4$ instead of $1/2$. After $m$ application of the inductive argument we arrive to the conclusion of the theorem in $B(0,(1/2)^m)$. Of course the radius $(1/2)^m$ has no importance and one could easily modify the proof in order to arrive in $B(0,1/2)$ at the end, up to change the constants.  This finishes the proof of the theorem, in the case when $f\in L^2(B^+(0,1))$.

We finally  assume now that $f \in H^\ell(B^+(0,1))$ and we want to get some further regularity. We cannot simply take a derivative of the equation as in the interior case, because it would not preserve the Dirichlet boundary condition. Instead, we will use discrete derivatives again. Consider an horizontal direction  $h=\varepsilon e_i$ with $1\leq i\leq N-1$ and notice that $D_hu$ is a solution with $D_{-h}f$ as a second member. In other words, 
\begin{align}
         \int_{B}A\nabla^m D_hu \cdot \nabla^m \varphi \;dx =
          -\int_{B_k} D_{-h}f  \varphi \; dx, \quad \textnormal{for every } \varphi \in H^m_0(B_k). \label{weakStart2}
     \end{align}
Since $f\in H^1(B^+(0,1))$, we know that $D_{-h}f \in L^2$, uniformly in $h$. Applying all the reasoning as before to the function $D_hu$ leads to the estimate:
    $$\|D_h u\|_{H^{2m}(B^+(0,\frac{1}{2}))}\leq C(\|\nabla^m D_h u\|_{L^{2}(B^+(0,1))} + \|D_{-h}f\|_{L^2(B^+(0,1)}),$$
and this actually implies 
    $$\|\partial_i u\|_{H^{2m}(B^+(0,\frac{1}{2}))}\leq C(\|\nabla^m u\|_{L^{2}(B^+(0,1))} + \|\partial_i f\|_{L^2(B^+(0,1)}).$$
The above works for all $1\leq i \leq N-1$. In other words, to prove that $u\in H^{2m+1}$ we only miss the derivative $\partial^{2m+1}_Nu$. 
But now we come back to the decomposition already used in \eqref{decompNN},
\begin{eqnarray}
\mathscr{A}=(-1)^m a_{me_N, me_N} \partial^{2m}_N  + \mathscr{E} \notag
\end{eqnarray} 
where $\mathscr{E}$ contains at most $2m-1$ derivatives in the $e_N$ direction.  Taking a derivative in the direction $e_N$ yields,
 \begin{eqnarray}
(-1)^m a_{me_N, me_N} \partial^{2m+1}_N u   = \partial_N f- \partial_N\mathscr{E}(u),\label{decompB00}
\end{eqnarray} 
and since the second member lies in $L^2(B^+(0,1))$ we deduce that $\partial^{2m+1}_N u\in L^2$. This proves that $u\in H^{2m+1}(B^+(0,1/2))$ and actually,
   $$\| u\|_{H^{2m+1}(B^+(0,\frac{1}{2}))}\leq C(\|\nabla^m u\|_{L^{2}(B^+(0,1))} + \| f\|_{H^1(B^+(0,1))}).$$
By a successive application of this reasoning in a sequence of balls $(B_k)$ we finally arrive to the conclusion that 
 $$\| u\|_{H^{2m+\ell}(B^+(0,\frac{1}{2}))}\leq C(\|\nabla^m u\|_{L^{2}(B^+(0,1))} + \| f\|_{H^\ell(B^+(0,1))}),$$
 which finishes the proof of the theorem.
\end{proof}

The following corollary will be used in the next section.

 \begin{corollary}[Regularity with  flat boundary and $f=0$]\label{theoremederegularite3}
    Let $m\in \N^*$. Let $\mathscr{A}$ be an  operator of order $2m$ of the form \eqref{ellipticOperator} and satisfying the ellipticity condition \eqref{ellipticPartial}. Let  $u \in H^m(B(0,R))$ a weak solution for $ \mathscr{A}u=0$ in $B^+(0,R)$, and such that $u=0$ in $B(0,R)\cap \{x_N<0\}$.  Then $\nabla^m u \in L^\infty(B^+(0,R/2))$ and
    $$\|\nabla^m u\|_{L^\infty(B^+(0,\frac{R}{2}))}\leq C_0\|\nabla^m u\|_{L^{2}(B^+(0,R))},$$
 where $C_0>0$ is a constant that depends  on    $N$,  $\max_{\alpha,\beta} |a_{\alpha,\beta}|$,  and $m$. 
   \end{corollary}

\begin{proof}  Theorem \ref{theoremederegularite2}, applied to the rescaled function $u(Rx)$,  says that $\nabla^m u\in H^\ell (B^+(0,R/2))$  for all $\ell\in \N$ because here $f=0$. Then the Sobolev embedding theorem that says that, provided $2\ell>N$, then $H^\ell \hookrightarrow \mathscr{C}^{0,\alpha}$, and in particular $\nabla^m u\in L^\infty$. We also directly get the   control on the norms because the norms that appear in the statement of Theorem \ref{theoremederegularite2} are homogenous after scaling.
\end{proof}

\section{Energy decay for a Reifenberg-flat boundary}

In this section we prove an energy decay property for the solution of a Dirichlet problem in a Reifenberg-flat domain. We begin with the case of a flat boundary, and then we argue by compactness to transfer this decay to a Reifenberg-flat boundary which is  ``sufficiently flat".   As before we consider an operator of order $2m$ of the form \eqref{ellipticOperator}.  
We start with a decay property at the boundary, in the case when $f=0$.

  \begin{proposition}[\textnormal{Decay for a poly-harmonic function near a Reifenberg-flat boundary}]\label{propdecroissanceenergiedebut}
    Let $a \in (0,1/4)$, $b>0$, $r_0 \in (0, 1]$  be given. Then there exists $\varepsilon_0 \in (0, 1)$ such that for every  $(\varepsilon_0,r_0)-$Reifenberg flat domain $\Omega$, $x \in \partial \Omega$, $r\leq r_0$, whenever $u \in  H^m_{0}(\Omega)$ is a weak solution of $\mathscr{A}(u)=0$  in $\Omega\cap B(x,r)$ then
    \begin{align}\label{propdecroissanceenergiedebutformule}
          \int_{B(x,ar) }\vert \nabla^m u \vert^2 dx \leq  C a^{N-b} \int_{B(x,r)}\vert \nabla^m u \vert^2 dx,
    \end{align}
    where $C=C_0^2\omega_N$, and $C_0>0$ is the constant of Corollary  \ref{theoremederegularite3}. 
  \end{proposition}%
\begin{proof}
   We argue by contradiction. If the Proposition is false then there exists $a \in (0,1/4)$, $b>0$, $r_0\in (0,1)$ such that for all $n\in \N$, there exists an open set $\Omega_n$ which is $(1/(n+1),r_0)$-Reifenberg-flat and satisfies the following. There exists $x_n\in \partial \Omega_n$, $r_n\leq r_0$ and $u_n \in H^m_{0}(\Omega_n)$ such that $\mathscr{A}(u)=0$  in $\Omega\cap B(x_n,r_n)$  and 
     \begin{align}\label{energiedecroiboule}
        Ca^{N-b}\int_{ B(x_n,r_n)}\vert \nabla^m u_n \vert^2 dx <  \int_{ B(x_n,ar_n)}\vert \nabla^m u_n \vert^2 dx.
    \end{align}
By means of a translation and rotation, we reduce to the case $x_n =0$ for all $n \in \N$ and the hyperplan induced by the definition of Reifenberg flat domain of $\Omega_n$ is $\mathscr{P}_0:=\left\{x_N=0\right\}$.

We consider the normalized sequence 
\begin{align*}
v_n(x) :=r_n^{\frac{N}{2}-2m}\Vert \nabla^m u_n(x) \Vert_{L^2(\Omega_n \cap B(0,r_n))}^{-1}  u_n(r_n x).
\end{align*}
The normalisation  is fixed in such a way that 
$$\int_{B_1}|\nabla^m v_n(x)|^2 \;dx =1.$$
To lighten the notation, we will use the notation  $\Omega_n$  for $\frac{1}{r_n}\Omega_n$. The functions $v_n$ are still solutions of $\mathscr{A}(v_n)=0$ in $\Omega_n\cap B(0,1)$, we have $v_n=0$ a.e. in $B(0,1)\setminus \Omega_n$ and the sequence $(v_n)_{n \in \N}$ is bounded in $H^m(B(0,1))$. Moreover,  \eqref{energiedecroiboule} becomes
 \begin{align}\label{energiedecroiboule0}
        Ca^{N-b}<  \int_{ B(0,a)}\vert \nabla^m v_n \vert^2 dx,
    \end{align}
and $\Omega_n\cap B(0,1)$ converges to $B^+(0,1)$ for the Hausdorff distance, in the sense that 
$$d_\mathscr{H}(\partial \Omega_n\cap B(0,1),  \{x_N=0\}\cap B(0,1))\to 0.$$
 The idea is to pass to the limit and obtain a contradiction thanks to the regularity of poly-harmonic functions at a flat boundary (Theorem \ref{theoremederegularite2}).

By weak compactness, we can extract a subsequence, still denoted by ${v}_{n}$ which weakly converges to $v \in H^m(B(0,1))$ and such that $(\nabla^i {v}_{n})_{n \in \N}$ strongly converges to $\nabla^i v$ in $L^2(B(0,1))$ for all $i \in \llbracket 0, m-1 \rrbracket$. The main objective is to show the following three facts:
\begin{itemize}
    \item \textit{Claim $A_1$.} The limit function $v$ satisfies $v=0$ a.e. in $B(0,1)\cap \{x_N<0\}$.
 \item \textit{Claim $A_2$.} The limit function $v$ is a weak solution of $\mathscr{A}(v)=0$ in $B^+(0,1)$.
 \item \textit{Claim $A_3$.} The sequence $(v_{n})_{n\in \N}$ strongly converges to $v$ in $H^m(B(0,1/2))$.
 \end{itemize}

Let us finish the proof by assuming that $A_1$, $A_2$ and $A_3$ holds. Passing to the limit in  inequality \eqref{energiedecroiboule0} we obtain,  
 \begin{align}\label{energiedecroiboule00}
        Ca^{N-b}\leq  \int_{ B(0,a)}\vert \nabla^m v \vert^2 dx.
    \end{align}
 On the other hand,   since $\mathscr{A}{v}=0$ in $B^+(0,1)$ and $v=0$ a.e. in $B(0,1)\cap\{x_N<0\}$, we can apply  Corollary  \ref{theoremederegularite3} to $v$ which says  that  for all $a\leq 1/2$,
  $$\|\nabla^m v\|_{L^\infty(B(0,a))}\leq C_0\|\nabla^m v\|_{L^{2}(B(0,1))}.$$
  In particular,
  \begin{eqnarray}
  \int_{B(0,a)} |\nabla^m v|^2 \; dx& \leq & \|\nabla^m v\|_{L^\infty(B^+(0,a))}^2 \omega_N a^N \notag\\
  &\leq & C_0^2\omega_N a^N \int_{B(0,1)}|\nabla^m v|^2  \; dx  =  C_0^2\omega_N a^N. \label{contre1}
  \end{eqnarray}
But since  $C=C_0^2\omega_N$, using \eqref{energiedecroiboule00} we arrive to
$$   C_0^2\omega_N a^{N-b}\leq C_0^2\omega_N a^N ,$$
which is a  contradiction because     $a<1$ .

  So to finish the proof, we are left to prove   $A_1$, $A_2$ and $A_3$. Actually, the claims $A_1$ and $A_2$ are easy to prove. Only $A_3$ is more delicate.
  
\quad\textit{Proof of  $A_1$.}  This fact directly follows from the Hausdorff convergence of $\partial \Omega_n\cap B(0,1)$ to $\{x_N=0\}\cap B(0,1)$. Indeed, for any $x_0 \in B(0,1)\cap \{x_N<0\}$, there exists a ball $B(x_0,\varepsilon)$ such that $\overline{B}(x_0,\varepsilon)\subset  B(0,1)\cap \{x_N<0\}$. From the Hausdorff convergence of $\Omega_n$ to $B^+(0,1)$, we deduce that for $n$ large enough, $B(x_0,\varepsilon) \subset \Omega_n^c \cap B(0,1)$. But then  we know by assumption that $v_n=0$ on $B(x_0,\varepsilon)$. Since $v_n\to v$ strongly in $L^2$, we deduce that $v=0$ a.e. on $B(x_0,\varepsilon)$, and since $x_0$ is arbitrary, we conclude that $v=0$ a.e. on $B(0,1)\cap \{x_N<0\}$, and the claim $A_1$ follows.

\quad\textit{Proof of  $A_2$.}  Since  $\partial \Omega_n\cap B(0,1)$ converges  to $\{x_N=0\}\cap B(0,1)$, we know that for any compact set $K\subset B^+(0,1)$, there exist $N_0$ large enough such that $K\subset \Omega_n$ for all $n\geq N_0$. But then for any test function $\varphi$ supported on $K$ we have
 $$\int_{B(0,1)} A\nabla^m v_n \cdot \nabla^m \varphi \;dx = 0.$$
Then from the weak convergence of $v_n$ to $v$ in $H^m$ we can pass to the limit to obtain
 $$\int_{B(0,1)} A\nabla^m v \cdot \nabla^m \varphi \;dx = 0.$$
Since this holds true for any compact set $K\subset B^+(0,1)$, the claim $A_2$ follows.

\quad\textit{Proof of  $A_3$.}  We show now that the sequence $(v_{n})_{n\in \N}$ strongly converges to $v$ in $H^m(B(0,\rho))$ for a radius $\rho\in(1/2,1)$.
Let $\rho$ be such a radius. By lower semicontinuity with respect to weak convergence we already know that 
\begin{align*}   
 \int_{B(0,\rho)} \vert \nabla^m v \vert^2 dx \leq \underset{n \rightarrow +\infty}{\liminf} \int_{B(0,\rho)}\vert \nabla^m v_n \vert^2 dx.
\end{align*}
To get the strong convergence in $B(0,\rho)$, it is enough to prove the converse, namely 
\begin{align*}
\underset{n \longrightarrow +\infty}{\limsup}   \int_{B(0,\rho)}\vert \nabla^mv_n\vert^2 dx \leq  \int_{B(0,\rho)} \vert \nabla^m v \vert^2 dx.
\end{align*}

The idea is to modify $v$ in order to make it admissible as a test function for $v_n$. 

Let $\xi \in C^\infty_c(B(0,1))$ be a smooth cut-off function such  that $\xi=1$ on $B(0,\rho)$. Since the function $v$ belongs to $H^m(B(0,1))$ and $v=0$ a.e. in $B(0,1)\cap \{x_N<0\}$, it follows from the theory of stable domains (see \cite{GLM23}) that $\xi v$ belongs to $H^m_0(B^+(0,1))$. We deduce that there exist a sequence $(w_i)_{i\in \N}$ in $\mathscr{C}^\infty_c(B^+(0,1))$ which strongly converges to $\xi v$ in $H^m(\R^N)$. The main point is that for $i$ being fixed, we know from the Hausdorff convergence of $\Omega_n$ that for  all $n$ large enough, $w_i\in H^m_0(\Omega_n \cap B(0,1))$.

Now we pick $\tilde{\rho} \in (1/2, \rho)$ and consider another cut-off function $\eta_{\tilde{\rho} } \in \mathscr{C}^\infty_c(B(0,\rho))$ such that $\eta_{\tilde{\rho} } =1$ on $B(0,\tilde{\rho} )$ and
\begin{align*}
    \vert \nabla^k \eta_{\tilde{\rho} } \vert \leq C(\rho-\tilde{\rho} )^{-k}, \quad 1 \leq k \leq m,
\end{align*}
for a universal constant $C>0$.

Since $v_n$ is a weak solution for $\mathscr{A}(v_n)=0$ in $B(0,\rho)\cap \Omega_n$,  it is a local energy minimizer and $z_{\tilde{\rho},i}:=(1-\eta_{\tilde{\rho}} )v_n+ \eta_{\tilde{\rho}}w_i$  is a competitor. We deduce that 
$$\int_{B(0,1)}A\nabla^m v_n \cdot \nabla^m v_n \;dx \leq  \int_{B(0,1)}A \nabla^m z_{\tilde{\rho},i} \cdot \nabla^m z_{\tilde{\rho},i} \;dx.$$

\begin{align}\label{inevnetzadmissible}
    \int_{B(0,\rho)}\vert \nabla^m v_n \vert^2 dx \leq \int_{B(0,\rho)} \vert \nabla^m z_{\tilde{\rho},i} \vert^2 dx.
\end{align}
 Next, a computation reveals that 
\begin{align*}
    \vert \nabla^m z_{\tilde{\rho} ,i}\vert^2 &= \sum_{\vert\alpha_1\vert=m} \left( \partial^{\alpha_1}(v_n + \eta_{\tilde{\rho} }(w_i - v_n))\right)^2 \\
    &=\sum_{\vert\alpha_1\vert=m}\left(\partial^{\alpha_1}v_n+ \sum_{\alpha_1 = \alpha_2 + \alpha_3}\frac{\alpha_1!}{\alpha_2! \alpha_3!}\partial^{\alpha_2} \eta_{\tilde{\rho} }\partial^{\alpha_3}(w_i -v_n)\right)^2 \\
      &=\sum_{\vert\alpha_1\vert=m}\left(\partial^{\alpha_1}v_n 
      +\eta_{\tilde{\rho} }\partial^{\alpha_1}(w_i -v_n)
      +\sum_{\substack{\alpha_1 = \alpha_2 + \alpha_3 \\ \vert \alpha_2  \vert>0}}\frac{\alpha_1!}{\alpha_2! \alpha_3!}\partial^{\alpha_2} \eta_{\tilde{\rho} }\partial^{\alpha_3}(w_i -v_n)\right)^2 \\
    &= \sum_{\vert\alpha_1\vert=m}\left(\eta_{\tilde{\rho} } \partial^{\alpha_1} w_i + (1-\eta_{\tilde{\rho} })\partial^{\alpha_1}v_n + \sum_{\substack{\alpha_1 = \alpha_2 + \alpha_3 \\ \vert \alpha_2  \vert>0}}\frac{\alpha_1!}{\alpha_2! \alpha_3!}\partial^{\alpha_2} \eta_{\tilde{\rho} }\partial^{\alpha_3}(w_i -v_n)\right)^2.
    \end{align*}
   By convexity, we get
\begin{align*}
     \vert \nabla^m z_{\tilde{\rho},i}\vert^2 &\leq \eta_{\tilde{\rho} }\vert\nabla^m w_i\vert^2 + (1-\eta_{\tilde{\rho} })\vert\nabla^mv_n\vert^2 + C(m)\sum_{k=1}^{m}\frac{1}{(\rho-\tilde{\rho} )^{2k}}\vert \nabla^{m-k}(w_i - v_n)\vert^2 \\
    & \quad \quad+ C(m)\sum_{k=1}^{m}\frac{1}{(\rho-\tilde{\rho} )^{2k}}\left|\left[\eta_{\tilde{\rho} } \nabla^m w_i + (1-\eta_{\tilde{\rho} })\nabla^mv_n \right]: \left[ \nabla^{m-k}(w_i - v_n)\right]\right|.
\end{align*}
Using \eqref{inevnetzadmissible}, one has
\begin{align*}
     \int_{B(0,\rho)} \eta_{\tilde{\rho}} \vert \nabla^mv_n \vert^2 dx &\leq \int_{B(0,\rho)}\eta_{\tilde{\rho} }  \vert  \nabla^m w_i \vert^2 dx + C(m)\sum_{k=1}^{m}\frac{1}{(\rho-\tilde{\rho} )^{2k}} \int_{B(0,\rho)} \vert \nabla^{m-k} (w_i - v_n)\vert^2 dx\\
    &+ C(m)\sum_{k=1}^{m}\frac{1}{(\rho-\tilde{\rho} )^{2k}}\int_{B(0,\rho)} \left| \left[\eta_{\tilde{\rho} } \nabla^m w_i + (1-\eta_{\tilde{\rho} })\nabla^mv_n \right]: \left[ \nabla^{m-k}(w_i - v_n)\right] \right|\;dx.
\end{align*}
Remind that $(\nabla^k v_n)_{n\in \N}$ strongly converge in $L^2(B(0,1))$ for all $k \in \llbracket 0, m-1 \rrbracket$, thus taking the limsup as $n \to + \infty$ in the above inequality we get 

\begin{align*}
 \limsup_{n\to +\infty}    \int_{B(0,\tilde{\rho})} \vert \nabla^mv_n \vert^2 dx  &\leq   \limsup_{n\to +\infty}    \int_{B(0,\rho)} \eta_{\tilde{\rho}} \vert \nabla^mv_n \vert^2 dx \notag \\
 &\leq \int_{B(0,\rho)}\eta_{\tilde{\rho} }  \vert  \nabla^m w_i \vert^2 dx + C(m)\sum_{k=1}^{m}\frac{1}{(\rho-\tilde{\rho} )^{2k}} \int_{B(0,\rho)} \vert \nabla^{m-k} (w_i - v)\vert^2 dx\\
    &+ C(m)\sum_{k=1}^{m}\frac{1}{(\rho-\tilde{\rho} )^{2k}}\int_{B(0,\rho)} \left[\eta_{\tilde{\rho} } \nabla^m w_i + (1-\eta_{\tilde{\rho} })\nabla^mv \right]: \left[ \nabla^{m-k}(w_i - v)\right] \;dx.
\end{align*}

Then we take the limit as $i \to + \infty$, and we recall that $w_i$ converges strongly in $H^m$ to $\xi v$, and $\xi v$ is equal to $v$ on $B(0,\rho)$. Therefore, all the terms with $w_i-v$ disappear and we arrive to 
\begin{align*}
 \limsup_{n\to +\infty}    \int_{B(0,\tilde{\rho})} \vert \nabla^m v_n \vert^2 dx &\leq  \int_{B(0,\rho)}\eta_{\tilde{\rho} }  \vert  \nabla^m v \vert^2 dx \leq  \int_{B(0,\rho)}   \vert  \nabla^m v \vert^2 dx.
 \end{align*}
Finally, by letting $\rho \to \rho'$ on the right hand side, we get the desired limsup inequality. This together with the liminf inequality and the weak-convergence, proves that $v_n$ converges strongly to $v$ in $H^m(B(0,\tilde{\rho}))$. Since $\tilde{\rho}>1/2$, this achieves the proof of Claim $A_3$, and the Proposition follows. 
\end{proof} 

We are now in position to prove a boundary decay property of a Dirichlet solution in a Reifenberg-flat domain.

\begin{proposition}[Energy decay for a solution at a boundary point]\label{decroissanceenergiereifenbergplat}
    Let $q \geq 2$ be such that $mq\geq N$ if $2m< N$, and $q=2$ otherwise. Let  $\eta \in (0,1)$ be given. Then there exists $\varepsilon_0 \in (0,1) $  such that for all $(\varepsilon_0,r_0)-$Reifenberg flat domain $\Omega$, for all $f\in L^q(\Omega)$, the weak solution $u \in H^m_0(\Omega)$ of $\mathscr{A}(u)=f$ satisfies
    \begin{align*}
        \int_{B(x,r) }\vert \nabla^m u \vert^2 dx \leq C   r^{N- \eta}\Vert f\Vert^2_{L^q(\Omega)}, \quad \forall r\in(0,r_0), \quad x\in \partial \Omega,
    \end{align*}
    where $C:= C(A,N,m,r_0,\eta, \Omega)>0$.
  \end{proposition}

\begin{proof}
    Consider a real $0<b< \eta$ and fix $a\in (0,1)$ satisfying
    \begin{align}
      a  \leq \left( \frac{1}{4C_A \max A} \right)^{\frac{1}{\eta-b}} , \label{defa}
    \end{align}
where $C_A$ is the ellipticity constant of $A$. 
    From Proposition \textit{\ref{propdecroissanceenergiedebut}}, for these choices of $a$ and $b$, there exists $\varepsilon_0\in(0,1)$  such that the decay property \eqref{propdecroissanceenergiedebutformule} holds for any poly-harmonic function.

 Let $u\in H^m_0(\Omega)$ be the weak solution of $\mathscr{A}(u)=f$ in $\Omega$ and let $x\in \partial \Omega$. and for every $k\in \N$, we define $B_k := B(x,a^kr_0)$. Let $v_k$ be the poly-harmonic replacement of $u$ in $B_k$, more precisely $v_k$ is the solution of
 $$\min_{w \text{ s.t. }w-u \in H^m_0(B_k\cap \Omega)} \int_{B_k}A \nabla^m w \cdot \nabla^m w \;dx.$$
 Then Proposition \textit{\ref{propdecroissanceenergiedebut}} applies to $v_k$ and yields
  \begin{align}
          \int_{B_{k+1} }\vert \nabla^m v_k \vert^2 dx \leq  C_1 a^{N-b} \int_{B_k}\vert \nabla^m v_k \vert^2 dx.
    \end{align}

We will prove by induction over $k  \geq k_0$ with $k_0:=\lfloor \frac{N}{\eta} \rfloor +1$ the inequality
    \begin{align}\label{recurrenceaveca}\tag{$H_k$}
        \int_{B_k}\vert \nabla^m u \vert^2 dx \leq C_{I} a^{k(N-\eta)}\Vert f \Vert_{L^q(\Omega)}^2, 
         \end{align}
        for a constant $C_I>0$ that will be fixed later.   
        
         First of all we have, by ellipticity of $A$,
    \begin{align*}
           \int_{B_{k_0} }\vert \nabla^m u \vert^2 dx \leq  \int_{\Omega}\vert \nabla^m u \vert^2 dx \leq C_A  \int_{\Omega} A \nabla^m u \cdot \nabla^m u \;dx = C_A\int_\Omega fu \; dx \leq \Vert f\Vert_{L^2(\Omega)} \Vert u\Vert_{L^2(\Omega)}.
    \end{align*}
    Next by the Poincar\'e inequality in $H^m_0(\Omega)$,
  $$\Vert u\Vert_{L^2(\Omega)} \leq C_P \|\nabla^m u\|_{L^2(\Omega)},$$  
    so finally, 
    \begin{eqnarray}
        \int_{B_{k_0} }\vert \nabla^m u \vert^2 dx \leq C_AC_P^2 \|f\|_{L^2(\Omega)}^2\leq C_AC_P|\Omega|^{1-\frac{2}{q}} \|f\|_{L^q(\Omega)}^2. \label{induction09}
    \end{eqnarray}
Hence, \eqref{recurrenceaveca} holds true for $k_0=\lfloor \frac{N}{\eta} \rfloor +1$ provided that 
$$C_AC_P|\Omega|^{1-\frac{2}{q}} a^{-k_0(N-\eta)} \leq C_I.$$

We assume now that \eqref{recurrenceaveca}  holds for some $k \geq k_0$ and we write
    \begin{align}
    \int_{B_{k+1}}\vert \nabla^m u \vert^2 dx &\leq 2  \int_{B_{k+1}}\vert \nabla^m v_{k} \vert^2 dx + 2 \int_{B_{k+1}}\vert \nabla^m (u-v_{k}) \vert^2 dx  \notag \\
        & \leq 2 a^{N-b} \int_{B_{k}}\vert \nabla^m v_{k} \vert^2 dx + 2 \int_{B_{k+1}}\vert \nabla^m (u-v_{k}) \vert^2 dx \notag \\
        & \leq C_A 2 a^{N- b} \int_{B_{k}} A \nabla^m u \cdot \nabla^m u \; dx + 2 \int_{B_{k}}\vert \nabla^m (u-v_{k}) \vert^2 dx.\notag\\
    & \leq \max(A) C_A 2 a^{N- b} \int_{B_{k}} | \nabla^m u|^2 \; dx + 2C_A \int_{B_{k}} A \nabla^m (u-v_{k}) \cdot \nabla^m (u-v_{k}) \; dx. \label{setapge}
    \end{align}
   The last inequality is a consequence of the fact that   $u$ is a competitor for the minimization problem satisfied by $v_k$. Moreover,   since $u-v_k$ is $A$-orthogonal to $v_k$, denoting by $|v |_A:=Av \cdot v$,  Pythagoras inequality yields
    \begin{align}\label{egalitepytagore11}
        0 \leq \int_{B_{k}}\vert \nabla^m (u-v_{k}) \vert^2_A dx =  \int_{B_{k}}\vert \nabla^m u \vert^2_A dx - \int_{B_{k}}\vert \nabla^m v_{k} \vert^2_A dx.
    \end{align}
    On the other hand, since now $u$ is a minimizer of
    \begin{align*}
        \textnormal{min}\left\{ \frac{1}{2}\int_{B_k} |\nabla^m w|_A dx - \int_{B_k}fw \; dx \; \middle| \; w \in u + H^m_0(\Omega\cap B_k)\right\},
    \end{align*}
and since the function $v_{k}$ is a competitor for this minimising problem we deduce that 
    \begin{align*}
        \int_{B_k}\vert \nabla^m u \vert^2_A dx - 2 \int_{B_k} fu \;dx \leq \int_{B_k}\vert \nabla^m v_{k} \vert^2_A dx - 2\int_{B_k} fv_{k} \;dx,
    \end{align*}
   and from \eqref{egalitepytagore11},
    \begin{align}
            \int_{B_{k}}\vert \nabla^m (u-v_{k}) \vert^2_A \; dx \leq 2 \int_{B_k} f(u-v_{k}) \;dx. \label{avantSobolev}
                \end{align}
  Then we use  the Sobolev embedding theorem  together with the Poincar\'e inequality (Lemma~\ref{poincare}) which   implies $u-v_k \in L^p$ with 
  \begin{eqnarray}
  p=
  \left\{\begin{array}{ll}
  \frac{2N}{N-2m} & \text{ if } 2m<N \\
  \text{ arbitrary large } &  \text{ if } 2m=N\\
  +\infty & \text{ if } 2m>N,
  \end{array}
  \right. \label{defer}
  \end{eqnarray}
and   moreover,  
  $$\| u-v_k\|_{L^p(B_k)}  \leq C_S(r)  \|\nabla^m(u-v_k) \|_{L^2(B_k)} \leq  C_S(p) \sqrt{C_A} \left( \int_{B_k} |\nabla^m(u-v_k)|_A^2 \;dx\right)^{\frac{1}{2}}.$$          
Hence, returning back to \eqref{avantSobolev} and using  H\"older inequality,
         \begin{eqnarray}
            \int_{B_{k}}\vert \nabla^m (u-v_{k}) \vert^2_A \; dx &\leq& 2 \int_{B_k} f(u-v_{k}) \;dx  \notag \\
            &\leq& 2 \|f\|_{L^{p'}(B_k)} \|u-v_k\|_p \notag \\
            &\leq &   2C_S(p) \sqrt{C_A}\|f\|_{L^{p'}(B_k)}   \left( \int_{B_k} |\nabla^m(u-v_k)|_A^2 \;dx\right)^{\frac{1}{2}}, \notag
         \end{eqnarray}
   which finally yields,
   \begin{eqnarray}
            \int_{B_{k}}\vert \nabla^m (u-v_{k}) \vert^2_A \; dx \leq    4C_S(p)^2C_A \|f\|_{L^{p'}(B_k)}^2.  \label{encoreO}         \end{eqnarray}
  Here, $p'\geq 1$ is the conjugate exponent of $p$, defined in \eqref{defer}, so that
   \begin{eqnarray}
  p'=
  \left\{\begin{array}{ll}
  \frac{2N}{N+2m} & \text{ if } 2m<N \\
  \text{ arbitrary close to  } 1 &  \text{ if } 2m=N\\
1& \text{ if } 2m>N.
  \end{array}
  \right. \label{deferP}
  \end{eqnarray}

In particular, under the assumptions of the proposition we know that $q\geq p'$. Let $s$ be such that $sp'=q$ (with $s=+\infty$ if $q=+\infty$). Then

$$\|f\|_{L^{p'}(B_k)}^{p'}\leq  |B_k|^{\frac{1}{s'}}\left( \int_{B_k} |f|^q \;dx\right)^{\frac{p'}{q}}, $$
and \eqref{encoreO} gives
   \begin{eqnarray}
            \int_{B_{k}}\vert \nabla^m (u-v_{k}) \vert^2_A \; dx &\leq&    4C_S(p)^2C_A C_N (a^k r_0)^\frac{2N}{p' s'} \|f\|_{L^{q}(\Omega)}^{2}.  \notag \\
           &=& 4C_S(p)^2C_A C_N (a^k r_0)^{2N \frac{q-p'}{p' q}} \|f\|_{L^{q}(\Omega)}^{2}.  \notag
                     \end{eqnarray}

Returning back to \eqref{setapge} we have obtain,
       \begin{align}
          \int_{B_{k+1}}\vert \nabla^mu \vert^2 dx &\leq  \max(A) C_A 2 C_I a^{N- b}a^{k(N-\eta)}\Vert f \Vert_{L^q(\Omega)}^2 \notag \\
          &+8C_A^2 C_S(p)^2  C_N (a^k r_0)^{2N \frac{q-p'}{p' q}} \|f\|_{L^{q}(\Omega)}^{2}. \label{2terms}
             \end{align}
      Let us check that each of the term in the sum on the right hand side above is less than $\frac{1}{2}a^{(k+1)(N-\eta)}C_I\Vert f \Vert_{L^q(\Omega)}^2$, which would be enough to finish the proof of   \eqref{recurrenceaveca}.
      
      For the first term it will be true provided that 
      $$ \max(A) C_A 2  a^{N- b}  \leq \frac{1}{2}a^{ N-\eta},$$
      which holds true thanks to our definition of $a$ (see \eqref{defa}).   
      
      For the second term we notice that $a$ has the following exponent:
$$ \gamma:=k2N \frac{q-p'}{p' q},$$
and we want this to be greater than  $(k+1)(N-\eta)$.  Then using that $k\geq k_0\geq \frac{N}{\eta}$ we estimate
\begin{eqnarray}
\gamma-(k+1)(N-\eta)&=&k2N \frac{q-p'}{qp'} -(k+1)(N-\eta) \notag \\
&=& \frac{1}{qp'} \left(2Nkq-k2Np'-Nkqp' +qp'(k\eta-N)+qp'\eta\right) \notag\\
&\geq &  \frac{1}{qp'} \left(2Nkq-k2Np'-Nkqp' \right) + \eta \notag \\
&\geq & \eta,
\end{eqnarray}
provided that $2Nkq-k2Np'-Nkqp'\geq 0$ or equivalently that 
\begin{eqnarray}
q(2-p') \geq 2p'. \label{conditionI}
\end{eqnarray} 
Then we consider three cases:

{\emph{Case 1.} If $2m>N$,} Then $p'=1$ and the condition in \eqref{conditionI} reduces to $q\geq 2$,  which is true.

{\emph{Case 2.}  If $2m=N$,} Then $p'$ can be chosen arbitrary close to $1$ thus we take $p'$ such that 
$$p'\leq \frac{2q}{1+q},$$
and then \eqref{conditionI} is true.

{\emph{Case 3.} If $2m<N$,} Then $p'= \frac{2N}{N+2m}$.  Thus a simple computation reveals that  \eqref{conditionI} is true   when $mq\geq N$, which is what we have assumed in the statement of the proposition.

{\emph{Conclusion.}} In all cases we have proved that the last term in \eqref{2terms} is always bounded by
$$ 
8C_A^2 C_S(p)^2  C_N r_0^{2N \frac{q-p'}{qp'}} a^{\eta} a^{(k+1)(N-\eta)}\|f\|_{L^{q}(\Omega)}^{2},$$
which is less than $\frac{1}{2}a^{(k+1)(N-\eta)}C_I\Vert f \Vert_{L^q(\Omega)}^2$ provided that 
$$8C_A^2 C_S(p)^2  C_N r_0^{2N \frac{q-p'}{qp'}} a^{\eta} \leq \frac{C_I}{2},$$
which can be  true by  defining well  $C_I$.  This finishes the proof of  \eqref{recurrenceaveca}.

 Now to finish the proof,  let $r\in (0, a^{k_0} r_0)$.  We can choose $k \geq  k_0$ such that $r < a^{k}r_0$ and $a^{k+1}r_0 \leq r$. One has
    \begin{align*}
           \int_{B(x,r)  }\vert \nabla^m u \vert^2 dx &\leq    \int_{B(x,a^kr_0) }\vert \nabla^m u \vert^2 dx \\
           &\leq C_{I} a^{k(N-\eta)}\Vert f \Vert^2_{L^q(\Omega)} \\
           &\leq C_{I} \left( \frac{r}{a r_0}\right)^{N-\eta}\Vert f\Vert^2_{L^q(\Omega)}.
    \end{align*}
On the other hand if $r\in (a^{k_0}r_0, r_0)$,   we can argue as for the proof of \eqref{induction09} yielding

   \begin{eqnarray}
        \int_{B_{r} }\vert \nabla^m u \vert^2 dx \leq C_AC_P|\Omega|^{1-\frac{2}{q}} \|f\|_{L^q(\Omega)}^2 \leq C_AC_P|\Omega|^{1-\frac{2}{q}}\frac{1}{(a^{k_0}r_0)^{N-\eta}} r^{N-\eta}\|f\|_{L^q(\Omega)}^2,\label{induction090}
    \end{eqnarray}

   and the proof is completed with $C :=\max (C_{I}(a r_0)^{\eta - N},C_AC_P|\Omega|^{1-\frac{2}{q}}\frac{1}{(a^{k_0}r_0)^{N-\eta}})$.
\end{proof}

Following exactly the same proof as the one of Proposition \ref{decroissanceenergiereifenbergplat}, we can also obtain the same for interior points.

\begin{proposition}[Energy decay for a solution in the interior]\label{decroissanceenergiereifenbergplatINT}
    Let $q \geq 2$ be such that $mq\geq N$ if $2m< N$, and $q=2$ otherwise. Let  $\eta \in (0,1)$ be given. Then there exists $r_0 \in (0,1)$ depending on $N$,  $\max_{\alpha,\beta} |a_{\alpha,\beta}|$,  and $m$  such that for all   $\Omega\subset \R^N$ open, for all $f\in L^q(\Omega)$,   the weak solution $u \in H^m_0(\Omega)$ of $\mathscr{A}(u)=f$ satisfies
    \begin{align*}
        \int_{B(x,r) }\vert \nabla^m u \vert^2 dx \leq C   r^{N- \eta}\Vert f\Vert^2_{L^q(\Omega)}, \quad \forall r\in(0,r_0), \quad x\in  \Omega, \text{ such that } B(x,r)\subset \Omega,
    \end{align*}
    where $C:= C(A,N,m,r_0,\eta, \Omega)>0$.
  \end{proposition}
  
  \begin{proof} The proof is exactly the same as Proposition \ref{decroissanceenergiereifenbergplat} using a polyharmonic replacement of $u$ in $B(x,r)$ to obtain some decay property on the energy. The only difference is that we use the interior decay property of a polyharmonic function stated in Corollary \ref{theoremederegularite3INT} instead of the Boundary decay of Proposition \ref{propdecroissanceenergiedebut} that was used in the proof of Proposition \ref{decroissanceenergiereifenbergplat}.
  \end{proof}

\section{Boundary regularity in a Reifenberg flat domain}

 This very short section contains the proof of our main regularity result, namely Theorem \ref{main}. The strategy is to use the energy decay of Proposition \ref{decroissanceenergiereifenbergplat} and Proposition \ref{decroissanceenergiereifenbergplatINT} of the derivative of the solution   to estimate its Campanato space norm. Here is our main result.

  \begin{theorem}\label{regularitereifenbergplat}
   Let $\alpha \in (0,1)$ and $q\geq 2$ be such that $mq\geq N$ if $2m<N$, and $q=2$ otherwise. There exists $\varepsilon_0 \in (0,1)$ and $r_0 \in (0,1]$ such that for every $(\varepsilon_0,r_0)-$Reifenberg-flat domain $\Omega\subset \R^N$, for every function $f\in L^q(\Omega)$, if $u \in H^m_0(\Omega)$ is the weak solution of $\mathscr{A}(u) = f$, then $u\in \mathscr{C}^{m-1,\alpha}(\overline{\Omega})$, and 
    $$\|u\|_{\mathscr{C}^{m-1,\alpha}(\overline{\Omega})} \leq C  \|f\|_{L^q(\Omega)},$$
    where $C>0$ depends on $N$,  $A$,  $\alpha$, $\Omega$ and $m$.
     \end{theorem}

\begin{proof}
    Using the energy decay of Proposition \textit{\ref{decroissanceenergiereifenbergplat}}, for every $\eta \in (0,1)$, there exists $\varepsilon_0 \in (0,1)$ and $r_0 \in (0,1]$ such that for all $y\in \partial\Omega$ and $r\in(0,r_0)$,
    \begin{align}\label{ontheboundary}
         \int_{\Omega \cap B(y,r)}\vert \nabla^m u \vert^2 dx \leq C r^{N- \eta}\Vert f\Vert^2_{L^q(\Omega)}.
    \end{align}    
Furthermore, Proposition \textit{\ref{decroissanceenergiereifenbergplatINT}} yields   
    \begin{align}\label{interiorEE}
        \int_{B(x,r) }\vert \nabla^m u \vert^2 dx \leq C   r^{N- \eta}\Vert f\Vert^2_{L^q(\Omega)}, \quad \forall r\in(0,r_1), \quad x\in  \Omega, \text{ such that } B(x,r)\subset \Omega.
    \end{align}
  All together, for any ball $B(x,r)$ with $x\in \Omega$ and $r\leq r_2:=\min(r_0,r_1)$, applying either \eqref{interiorEE} if $B(x,r)\subset \Omega$ or \eqref{ontheboundary} in $B(z,2r)$ with $z\in \partial\Omega$ such that $B(x,r)\subset B(z,2r)$ if $B(x,r)\cap \partial \Omega\not =\emptyset$, we deduce that 
  \begin{align}\label{interiorEEE}
        \int_{B(x,r) }\vert \nabla^m u \vert^2 dx \leq C   r^{N- \eta}\Vert f\Vert^2_{L^q(\Omega)}, \quad \forall r\in(0,r_2), \quad x\in  \Omega.
    \end{align}

By  Proposition \ref{continuitedecroissance}, we get $u \in\mathscr{C}^{m-1, \alpha}(\overline{\Omega})$  with $\alpha := 1 - \eta/2$ and 
 $$\|u\|_{\mathscr{C}^{m-1,\alpha}(\overline{\Omega})} \leq C  \|f\|_{L^q(\Omega)},$$
and the Theorem follows.
\end{proof}

\bibliography{biblio}
\bibliographystyle{plain}

 \end{document}